\documentclass[12pt]{amsart}

\usepackage{a4wide}
\usepackage{color}
\usepackage{amssymb, amsmath}
\usepackage{mathrsfs}

\usepackage{float}
\usepackage{graphicx}

\usepackage{tikz}
\usepackage{tikz-cd}
\usetikzlibrary{snakes}
\usetikzlibrary{intersections, calc}
\usepackage{epsfig}
\usepackage[all]{xy}
\usepackage{epstopdf}
\usepackage{framed}

\newtheorem{theorem}{Theorem}[section]
\newtheorem{lemma}[theorem]{Lemma}
\newtheorem{proposition}[theorem]{Proposition}
\newtheorem{corollary}[theorem]{Corollary}

\theoremstyle{remark}
\newtheorem{remark}{Remark}[section]
\newtheorem{definition}{Definition}[section]
\newtheorem{example}{Example}[section]

\newtheorem{relation}{Relation}

\numberwithin{equation}{section}

\newcommand{\algt}{\mathfrak{t}}

\begin{document}

\title[Equivariant cohomology of even-dimensional complex quadrics]{Equivariant cohomology of even-dimensional complex quadrics from a combinatorial point of view}

\author[S.\ Kuroki]{Shintar\^o Kuroki}
\address{Department of Applied Mathematics Faculty of Science, Okayama University of Science, 1-1 Ridai-Cho Kita-Ku Okayama-shi Okayama 700-0005, Okayama, Japan}
\email{kuroki@ous.ac.jp}

\dedicatory{To the memory of Professor Fuichi Uchida.}


\begin{abstract}
The purpose of this paper is to determine the ring structure of the graph equivariant cohomology of the GKM graph induced from the complex quadrics $Q_{2n}$.
We show that the graph equivariant cohomology is generated by two types of subgraphs in the GKM graph, namely $M_{v}$ and $\Delta_{K}$, 
which are subject to four different types of relations.
By utilizing this ring structure, we establish the multiplicative relation for the generators $\Delta_{K}$ of degree $2n$ and provide an alternative computation of the ordinary cohomology ring of $Q_{2n}$, as previously computed by H.~Lai.
Additionally, we provide a combinatorial explanation for why the square of the half-degree generator $x\in H^{2n}(Q_{2n})$ vanishes when $n$ is odd and is non-vanishing when $n$ is even.
\end{abstract}

\maketitle

\section{Introduction}
\label{sect:1}

In the paper \cite{GKM}, Goresky, Kottwiz, and MacPherson 
established a framework for studying the class of manifolds with a torus action, known as {\it equivariantly formal}, 
by using their fixed points and one-dimensional orbits. 
These manifolds are now commonly referred to as {\it GKM manifold}.
Expanding on their work, Guillemin and Zara introduced the notion of an abstract {\it GKM graph} in \cite{GZ01} as a combinatorial counterpart of GKM manifolds, thus initiating the study of spaces with torus actions using the combinatorial structure of GKM graphs.
Since then, 
the research of GKM manifolds and GKM graphs, commonly known as GKM theory, has been the subject of extensive research (e.g., \cite{GHZ, GKZ, GSZ,  Ku16, Ku19, MMP}).

One can view GKM theory as a methodology for computing equivariant cohomology based on the combinatorial structure of a graph.
For an equivariantly formal GKM manifold, 
its equivariant cohomology is isomorphic to the {\it graph equivariant cohomology} of its corresponding GKM graph, see \eqref{graph-equivariant-cohomology}. 
On the other hand, for abstract GKM graphs, the graph equivariant cohomology can be defined independently of geometry, leading to its study in various articles 
(e.g., \cite{DKS, FY, FIM, GHZ, GSZ, KU, KS, MMP}).
In particular, in \cite{MMP}, Maeda-Masuda-Panov introduced the combinatorial counterpart of a torus manifold, where a {\it torus manifold} is defined by a $2n$-dimensional $T^{n}$-manifold with fixed points.
This combinatorial object is called a {\it torus graph}, and 
its properties have been extensively studied. 
Notably, they established that the graph equivariant cohomology of a torus graph is isomorphic to its {\it face ring}, which is defined using the simplicial poset induced from the subgraphs of a torus graph, relying solely on algebraic and combinatorial arguments.
The advantage of establishing such a result for abstract GKM graphs, without relying on geometry, is that it can be applied to 
a wider class of equivariantly formal GKM manifolds (or spaces) that share the same GKM graph.
This enables us to compute the equivariant cohomology of equivariantly formal GKM manifolds, even when well-known techniques for computing equivariant cohomology,  such as certain methods in algebraic topology or Hamiltonian torus actions, cannot be applied.
Hence, the result in \cite{MMP} can be regarded as a generalization of the computation of the equivariant cohomology ring of torus manifolds presented in \cite{MP}. 

In our paper, we focus on the study of GKM graphs corresponding to {\it even-dimensional complex quadrics}.
An even-dimensional complex quadric $Q_{2n}$ is defined by 
\begin{align*}
Q_{2n}:= \{[z_{1}:\cdots :z_{2n+2}]\in \mathbb{C}P^{2n+1}\ |\ \sum_{i=1}^{n+1}z_{i}z_{2n+3-i}=0\},
\end{align*}
having the natural $T^{n+1}$-action
\begin{align}
\label{def-noneff-action}
[z_{1}:\cdots :z_{2n+2}]\mapsto [z_{1}t_{1}: z_{2}t_{2}: \cdots z_{n+1}t_{n+1}: t_{n+1}^{-1}z_{n+2}: t_{n}^{-1}z_{n+3}: \cdots t_{1}^{-1}z_{2n+2}],
\end{align}
where $(t_{1},\ldots, t_{n+1})\in T^{n+1}$.
Since $Q_{2n}\simeq SO(2n+2)/SO(2n)\times SO(2)$, this action is equivalent to 
the restriction of the transitive $SO(2n+2)$-action to the maximal torus $T^{n+1}$-action.
As $T^{n+1}$ is also a maximal torus of $SO(2n)\times SO(2)$ (i.e., $SO(2n)\times SO(2)$ is a maximal rank subgroup of $SO(2n+2)$), it follows from \cite{GHZ} that the fixed points and one-dimensional orbits of the $T^{n+1}$-action have the structure of a graph.
Therefore, the GKM graph of $Q_{2n}$ with the $T^{n+1}$-action \eqref{def-noneff-action} can be constructed by labeling the edges with tangential representations.
Although the action \eqref{def-noneff-action} has a finite kernel $\mathbb{Z}_{2}=\{\pm 1\}\subset T^{n+1}$, 
we can obtain an effective $T^{n+1}$-action on $Q_{2n}$ by considering the quotient $T^{n+1}/\mathbb{Z}_{2}$.
In this paper, we denote by $\mathcal{GQ}_{2n}$ the GKM graph obtained from this effective $T^{n+1}$-action, see Section~\ref{sect:2.2}.

On the other hand, the ordinary cohomology ring $H^{*}(Q_{2n})$ of $Q_{2n}$ over the integer coefficient was computed by H.~Lai in \cite{La72, La74} (also see \cite[Excercise 68.3]{EKM08} for $H^{*}(Q_{m})$ as the Chow ring\footnote{Since $Q_{m}$ can also be regarded as the homogeneous space of the affine algebraic group $SO(m+2,\mathbb{C})$, it follows from \cite[Appendix C.3.4]{EH13} that its Chow ring is isomorphic to its cohomology ring, i.e., $A^{*}(Q_{m})\simeq H^{2*}(Q_{m};\mathbb{Z})$. We also note that the rational cohomology ring of $Q_{2n-1}$ is isomorphic to that of $\mathbb{C}P^{2n-1}$ (e.g. see \cite{Uc77}); however, these two cohomologies are not isomorphic over integer coefficients (e.g. see \cite{EKM08, Jo}).}).
In particular, we have the following isomorphisms.
\begin{align}
\label{ordinary-cohom}
H^{*}(Q_{m})\simeq 
\left\{
\begin{array}{lll}
\mathbb{Z}[c,x]/\langle c^{2n+1}-2cx, x^{2}-c^{2n}x \rangle &  {\rm if}\ m=4n, &  {\rm where}\ \deg c=2,\ \deg x=4n \\
\mathbb{Z}[c,x]/\langle c^{2n+2}-2cx, x^{2} \rangle &  {\rm if}\ m=4n+2, &  {\rm where}\ \deg c=2,\ \deg x=4n+2 
\end{array}
\right.
\end{align}
Using this formula, one can conclude that $H^{odd}(Q_{2n})=0$ which means that $Q_{2n}$ is an {\it equivariantly formal GKM manifold}.
Therefore, 
the equivariant cohomology $H_{T^{n+1}}^{*}(Q_{2n})$ of the effective $T^{n+1}$-action on $Q_{2n}$ can be computed by using the graph equivariant cohomology of its GKM graph, denoted by $\mathcal{GQ}_{2n}$.
The main goal of this paper is to determine the graph equivariant cohomology $H^{*}(\mathcal{GQ}_{2n})$ (see \eqref{graph-equivariant-cohomology}) by explicitly describing its generators and relations in terms of the subgraphs. 
As a consequence, we can compute the equivariant cohomology ring of the effective $T^{n+1}$-action on $Q_{2n}$ by generators and relations.
The main theorem of this paper, which is presented in Section~\ref{sect:5} precisely, 
is as follows: 
\begin{theorem}
\label{main}
There exist the following isomorphisms as a ring:
\begin{align*}
H^{*}_{T^{n+1}}(Q_{2n})\simeq H^{*}(\mathcal{GQ}_{2n})\simeq \mathbb{Z}[\mathcal{GQ}_{2n}].
\end{align*} 
\end{theorem}

Since the complex quadric $Q_{2n}$ is equivariantly formal, the Serre spectral sequence of the fiber bundle $Q_{2n}\to ET\times_{T}Q_{2n}\to BT$ collapses at the $E_{2}$-term.
This implies that 
the ordinary cohomology $H^{*}(Q_{2n})$ can be obtained as the quotient of the equivariant cohomology $H_{T^{n+1}}^{*}(Q_{2n})$ by $H^{>0}(BT)$.
It is worth noting that the ring structure of $H^{*}(Q_{2n})$, as shown in \eqref{ordinary-cohom}, depends on whether $n$ is even or odd.
We provide a combinatorial explanation for the difference between $H^{*}(Q_{4n})$ and $H^{*}(Q_{4n+2})$ using Theorem~\ref{main} (see Lemma~\ref{final-lemma} and  Corollary~\ref{final-cor} precisely). 

The paper is organized as follows, consisting of Sections~\ref{sect:2} through \ref{sect:7}.
In Section~\ref{sect:2}, 
we compute the GKM graph $\mathcal{GQ}_{2n}$ of the effective $T^{n+1}$-action on $Q_{2n}$.
In Section~\ref{sect:3}, 
we introduce the graph equivariant cohomology $H^{*}(\mathcal{GQ}_{2n})$ 
and define the generators $M_{v}$ and $\Delta_{K}$, studying their properties.
In Section~\ref{sect:4}, we present the four relations among $M_{v}$ and $\Delta_{K}$.
The main theorem (Theorem~\ref{main-theorem2}) is proved in Section~\ref{sect:5}.
Section~\ref{sect:6} and Section~\ref{sect:7} serve as additional sections with applications of Theorem~\ref{main}.
In Section~\ref{sect:6}, we establish multiplicative relations among $\Delta_{K}$'s of degree $2n$.
In Section~\ref{sect:7}, the ordinary cohomology ring of $Q_{2n}$ is studied from a GKM theoretical perspective.

\section{GKM graphs of even-dimensional complex quadrics $Q_{2n}$}
\label{sect:2}

In this section, 
we compute the GKM graph of the effective $T^{n+1}$-action on $Q_{2n}$ (see \cite{GZ01, Ku09} about the basic facts of the GKM graph). 
In this paper, we identify the cohomology ring $H^{*}(BT^{n+1})$ as the following polynomial ring generated by degree $2$ generators $x_{1},\ldots, x_{n+1}$:
\begin{align}
\label{generators_of_BT}
H^{*}(BT^{n+1})\simeq \mathbb{Z}[x_{1},\ldots, x_{n+1}].
\end{align}
It is worth noting that the generator $x_{i}$, for $i=1,\ldots, n+1$, is 
the equivariant first Chern class of the $T^{n+1}$-equivariant complex line bundle over a point, where the action on the unique fiber is defined by the $i$th coordinate projection $p_{i}:T^{n+1}\to S^{1}\in {\rm Hom}(T^{n+1},S^{1})$.
This gives the following identifications:
\begin{align*}
H^{2}(BT^{n+1})\simeq {\rm Hom}(T^{n+1},S^{1})\simeq(\algt_{\mathbb{Z}}^{n+1})^{*}\simeq \mathbb{Z}^{n+1},
\end{align*} 
where $\mathfrak{t}_{\mathbb{Z}}^{n+1}$ is the lattice of the Lie algebra of $T^{n+1}$.
In this paper, we often use this identification.

\subsection{The GKM graph of the natural $T^{n+1}$-action on $Q_{2n}$} 
\label{sect:2.1}

Suppose that the $T^{n+1}$-action on $Q_{2n}$ is defined by \eqref{def-noneff-action}.
We first compute the GKM graph of this non-effective $T^{n+1}$-action.

By definition, the GKM graph consists of the fixed points (vertices) and the invariant $2$-spheres (edges), and the labels on edges (the {\it axial function} of the GKM graph) which are defined by the tangential representations on fixed points.
It is easy to check from the definition \eqref{def-noneff-action} that  
the fixed points of $Q_{2n}$ are
\begin{align*}
Q_{2n}^{T}=\{[e_{i}]\ |\ i=1,\ldots, 2n+2\},
\end{align*}
where $[e_{i}]=[0:\cdots :0:1:0:\cdots :0]\in \mathbb{C}P^{2n+1}$ (only the $i$th coordinate is $1$).
We first denote the $2$-spheres in $\mathbb{C}P^{2n+1}$ by the following symbol:
\begin{align}
\label{invariant-sphere}
[z_{i}:z_{j}]:=\{[0:\cdots: 0: z_{i}:0: \cdots: 0: z_{j}: 0\cdots: 0 ]\in \mathbb{C}P^{2n+1}\},
\end{align} 
i.e., the subset consists of the only $i$th and $j$th coordinates. 
If $[z_{i}:z_{j}]\subset Q_{2n}$, then it follows from the quadric equation $\sum_{i=1}^{n+1}z_{i}z_{2n+3-i}=0$ which defines $Q_{2n}$ that one of the following properties hold:
\begin{itemize}
\item 
$[z_{i}:z_{j}]\simeq \mathbb{C}P^{1}$ (diffeomorphic) if $i+j\not=2n+3$;
\item
$[z_{i}:z_{j}]=\{[1:0]\}$ or $\{[0:1]\}$ if $i+j=2n+3$.
\end{itemize}
Namely, invariant $2$-spheres of $Q_{2n}$ are $[z_{i}:z_{j}]$ such that $i+j\not=2n+3$.
Therefore, we obtain the following graph from 
the $T^{n+1}$-action on $Q_{2n}$: 
\begin{itemize}
\item the set of vertices $V_{2n}=[2n+2]:=\{1,2,\ldots, 2n+2\}$;
\item the set of edges $E_{2n}=\{ij\ |\ i,j\in [2n+2] \ \text{such that}\ i\not=j,\ i+j\not=2n+3 \}$.
\end{itemize}
We denote this graph as $\Gamma_{2n}:=(V_{2n},E_{2n})$, see Figure~\ref{2-examples}. 
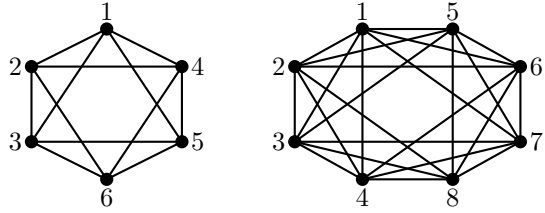
\begin{figure}[h]
\begin{tikzpicture}
\begin{scope}[xscale=0.5, yscale=0.5]

\fill(-3,2) circle (5pt);
\node[above] at (-3,2) {$1$};
\fill(-5,1) circle (5pt);
\node[left] at (-5,1) {$2$};
\fill(-5,-1) circle (5pt);
\node[left] at (-5,-1) {$3$};
\fill(-3,-2) circle (5pt);
\node[below] at (-3,-2) {$6$};
\fill(-1,-1) circle (5pt);
\node[right] at (-1,-1) {$5$};
\fill(-1,1) circle (5pt);
\node[right] at (-1,1) {$4$};

\draw[thick] (-3,2)--(-5,1);
\draw[thick] (-3,2)--(-5,-1);
\draw[thick] (-3,2)--(-1,-1);
\draw[thick] (-3,2)--(-1,1);
\draw[thick] (-5,1)--(-5,-1);
\draw[thick] (-5,1)--(-3,-2);
\draw[thick] (-5,1)--(-1,1);
\draw[thick] (-5,-1)--(-1,-1);
\draw[thick] (-5,-1)--(-3,-2);
\draw[thick] (-1,1)--(-1,-1);
\draw[thick] (-1,1)--(-3,-2);
\draw[thick] (-1,-1)--(-3,-2);

\fill(3.8,2) circle (5pt);
\node[above] at (3.8,2) {$1$};
\fill(2,1) circle (5pt);
\node[left] at (2,1) {$2$};
\fill(2,-1) circle (5pt);
\node[left] at (2,-1) {$3$};
\fill(3.8,-2) circle (5pt);
\node[below] at (3.8,-2) {$4$};
\fill(6.2,2) circle (5pt);
\node[above] at (6.2,2) {$5$};
\fill(8,-1) circle (5pt);
\node[right] at (8,-1) {$7$};
\fill(8,1) circle (5pt);
\node[right] at (8,1) {$6$};
\fill(6.2,-2) circle (5pt);
\node[below] at (6.2,-2) {$8$};

\draw[thick] (3.8,2)--(2,1);
\draw[thick] (3.8,2)--(2,-1);
\draw[thick] (3.8,2)--(3.8,-2);
\draw[thick] (3.8,2)--(6.2,2);
\draw[thick] (3.8,2)--(8,-1);
\draw[thick] (3.8,2)--(8,1);
\draw[thick] (6.2,-2)--(2,1);
\draw[thick] (6.2,-2)--(2,-1);
\draw[thick] (6.2,-2)--(3.8,-2);
\draw[thick] (6.2,-2)--(6.2,2);
\draw[thick] (6.2,-2)--(8,-1);
\draw[thick] (6.2,-2)--(8,1);
\draw[thick] (2,1)--(2,-1);
\draw[thick] (2,1)--(3.8,-2);
\draw[thick] (2,1)--(6.2,2);
\draw[thick] (2,1)--(8,1);
\draw[thick] (8,-1)--(2,-1);
\draw[thick] (8,-1)--(3.8,-2);
\draw[thick] (8,-1)--(6.2,2);
\draw[thick] (8,-1)--(8,1);
\draw[thick] (2,-1)--(3.8,-2);
\draw[thick] (2,-1)--(6.2,2);
\draw[thick] (8,1)--(3.8,-2);
\draw[thick] (8,1)--(6.2,2);
\end{scope}
\end{tikzpicture}
\caption{The left graph is $\Gamma_{4}$ ($n=2$) induced from the $T^{3}$-action on $Q_{4}$, and the right graph is $\Gamma_{6}$ ($n=3$) induced from the $T^{4}$-action on $Q_{6}$.}
\label{2-examples}
\end{figure}

\begin{remark}
For convenience, we often denote the vertex $j\in V_{2n}$ such that $i+j=2n+3$ by  $\overline{i}$.
Namely, the set of vertices can be written by
\begin{align*}
V_{2n}=[2n+2]=\{1,2,\ldots, n+1, \overline{n+1},\overline{n},\ldots, \overline{1}\}.
\end{align*}
Moreover, by using this notation, the set of edges can be written by
\begin{align*}
E_{2n}=\{ij\ |\ i,j\in V_{2n}\ \text{such that}\ j\not=i, \overline{i} \}.
\end{align*}
\end{remark}

We next compute the tangential representations around the fixed points and put the label on edges denoting as $\widetilde{\alpha}:E_{2n}\to H^{2}(BT^{n+1})$, called an {\it axial function} on edges. 
Recall that the tangential representations around the fixed points decompose into the complex $1$-dimensional irreducible representations.
One can also regard each complex $1$-dimensional irreducible representation as the tangential representation on the fixed point of the invariant $2$-sphere.
This implies that, 
to compute the tangential representations around fixed points, 
it is enough to compute the tangential representation on each invariant $2$-sphere $[z_{i}:z_{j}]\in Q_{2n}$, see \eqref{invariant-sphere}.
By the definition of the $T^{n+1}$-action on $[z_{i}:z_{j}]$, 
we may write 
the action $t=(t_{1},\ldots, t_{n+1})\in T^{n+1}$ on $[z_{i}:z_{j}]$ as 
\begin{align*}
[z_{i}:z_{j}]\mapsto [p_{i}(t)z_{i}:p_{j}(t)z_{j}], 
\end{align*}
 giving 
the $T^{n+1}$-actions on the two fixed points of the $2$-sphere $[z_{i}:z_{j}]$ by 
\begin{align*}
[1:z_{j}]\mapsto [1:p_{i}(t)^{-1}p_{j}(t)z_{j}], \quad 
[z_{i}:1]\mapsto [p_{i}(t)p_{j}(t)^{-1}z_{i}:1],
\end{align*}
where $p_{i}:T^{n+1}\to S^{1}$ is the surjective homomorphism defined as the following map.
\begin{align*}
p_{i}(t)
=\left\{
\begin{array}{ll}
t_{i} & \text{if $i\in [n+1]$} \\
t_{\overline{i}}^{-1} & \text{if $i\in \{n+2,\ldots, 2n+2\}$}
\end{array}
\right.
\end{align*}
Therefore, the axial function $\widetilde{\alpha}:E_{2n}\to H^{2}(BT^{n+1})$ is defined by the following equation (see Figure~\ref{fig_non-effective}):
\begin{align}
\label{non-eff_axial_function}
\widetilde{\alpha}(ij)=x_{j}-x_{i},
\end{align}
where $x_{i}\in H^{2}(BT^{n+1})$ is the element such that 
\begin{itemize}
\item for $i\in [n+1]$, $x_{i}$ is the generator of $H^{2}(BT^{n+1})$ corresponds to the $i$th coordinate projection $p_{i}$, also see \eqref{generators_of_BT};
\item for $i\in \{n+2,\ldots, 2n+2\}$, $x_{i}:=-x_{\overline{i}}$.
\end{itemize}

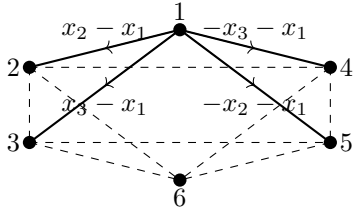
\begin{figure}[h]
\begin{tikzpicture}
\begin{scope}[xscale=0.5, yscale=0.5]

\fill(0,2) circle (5pt);
\node[above] at (0,2) {$1$};
\fill(-4,1) circle (5pt);
\node[left] at (-4,1) {$2$};
\fill(-4,-1) circle (5pt);
\node[left] at (-4,-1) {$3$};
\fill(0,-2) circle (5pt);
\node[below] at (0,-2) {$6$};
\fill(4,-1) circle (5pt);
\node[right] at (4,-1) {$5$};
\fill(4,1) circle (5pt);
\node[right] at (4,1) {$4$};

\draw[thick] (0,2)--(-4,1);
\draw[thick] (0,2)--(-4,-1);
\draw[thick] (0,2)--(4,-1);
\draw[thick] (0,2)--(4,1);
\draw[dashed] (-4,1)--(-4,-1);
\draw[dashed] (-4,1)--(0,-2);
\draw[dashed] (-4,1)--(4,1);
\draw[dashed] (-4,-1)--(4,-1);
\draw[dashed] (-4,-1)--(0,-2);
\draw[dashed] (4,1)--(4,-1);
\draw[dashed] (4,1)--(0,-2);
\draw[dashed] (4,-1)--(0,-2);

\node[above] at (-2,1.5) {$x_{2}-x_{1}$};
\draw[->] (0,2)--(-2,1.5);
\node[above] at (2,1.5) {$-x_{3}-x_{1}$};
\draw[->] (0,2)--(2,1.5);
\node[below] at (-2,0.5) {$x_{3}-x_{1}$};
\draw[->] (0,2)--(-2,0.5);
\node[below] at (2,0.5) {$-x_{2}-x_{1}$};
\draw[->] (0,2)--(2,0.5); 
\end{scope}
\end{tikzpicture}
\caption{The axial function $\widetilde{\alpha}$ around the vertex $1$ in $\Gamma_{4}$. This corresponds to the GKM graph induced from the $T^{3}$-action on $Q_{4}$ defined by \eqref{def-noneff-action}. Note that $\overline{6}=1$, $\overline{5}=2$, $\overline{4}=3$.}
\label{fig_non-effective}
\end{figure}

\subsection{The GKM graph of the effective $T^{n+1}$-action on $Q_{2n}$}
\label{sect:2.2}

Since the $T^{n+1}$-action \eqref{def-noneff-action} on $Q_{2n}$ is not effective,
the axial function $\widetilde{\alpha}$ defined by  
\eqref{non-eff_axial_function} 
does not satisfy the {\it effectiveness conditions}, i.e., 
for any fixed $i\in V_{2n}$, the set $\{\widetilde{\alpha}(ij)\ |\ ij\in E_{2n}\}$ does not span $H^{2}(BT^{n+1})\simeq \mathbb{Z}^{n+1}$ 
 (see \cite[Section 2.1]{Ku19}).
For example, around the vertex $1\in V_{2n}$, the axial functions are
\begin{align}
\label{non-primitive-vectors}
x_{2}-x_{1},\ \ldots,\ x_{n+1}-x_{1},\ -x_{n+1}-x_{1},\ -x_{n}-x_{1},\ \ldots,\ -x_{2}-x_{1}\in (\mathfrak{t}_{\mathbb{Z}}^{n+1})^{*},
\end{align}
and it is easy to check that 
these vectors span the lattice $\langle x_{2}-x_{1},\ldots, x_{n+1}-x_{1}, -x_{n+1}-x_{1}\rangle_{\mathbb{Z}}$ which is the proper subspace in $(\mathfrak{t}_{\mathbb{Z}}^{n+1})^{*}$. 
This is also similar to the axial functions on the other vertices. 
To apply the GKM theory, we will identify 
$\langle x_{2}-x_{1},\ldots, x_{n+1}-x_{1}, -x_{n+1}-x_{1}\rangle_{\mathbb{Z}}$ as 
$(\algt^{n+1}_{\mathbb{Z}})^{*}$.
In this paper, they are replaced as follows:
\begin{itemize}
\item $x_{i}-x_{1}$ as $x_{i-1}$ for $i=2,\ldots, n+1$;
\item $-x_{n+1}-x_{1}$ as $x_{n+1}$.
\end{itemize}
For the other vectors in \eqref{non-primitive-vectors},  we have the following equalities:
\begin{align*}
-x_{i}-x_{1}=-(x_{i}-x_{1})+(x_{n+1}-x_{1})+(-x_{n+1}-x_{1})
\end{align*}
for $i=2,\ldots, n$.
Therefore,  
we may replace the 
vectors in \eqref{non-primitive-vectors} with the following vectors (respectively):
\begin{align}
\label{basis_vectors}
x_{1},\ \ldots,\ x_{n},\ x_{n+1},\ -x_{n-1}+x_{n}+x_{n+1},\ \ldots,\ -x_{1}+x_{n}+x_{n+1}.
\end{align}
Notice that the vectors in \eqref{basis_vectors} are primitive generaters of $(\mathfrak{t}_{\mathbb{Z}}^{n+1})^{*}$.
This gives the axial function induced by the effective $T^{n+1}(\simeq T^{n+1}/\mathbb{Z}_{2})$-action on $Q_{2n}$, where $\mathbb{Z}_{2}=\{\pm 1\}$ is the kernel of the $T^{n+1}$-action in \eqref{def-noneff-action} (more precisely, see the following Remark~\ref{pre-ex}).

\begin{remark}
\label{pre-ex}
Here, we will explain that the axial function \eqref{basis_vectors} is obtained by the explicit $T^{n+1}$-action on $Q_{2n}\subset \mathbb{C}P^{2n+1}$ defined by \eqref{ef-ac}. 
The vectors \eqref{non-primitive-vectors} in $(\mathfrak{t}^{n+1}_{\mathbb{Z}})^{*}$ induce the non-injective homomorphism $\varphi:T^{n+1}\to T^{2n+2}$ with $\ker \varphi=\{\pm 1\}=\mathbb{Z}_{2}$ defined by 
\begin{align*}
(t_{1},\ldots, t_{n+1})\mapsto (1,t_{2}t_{1}^{-1},\ldots, t_{n+1}t_{1}^{-1}, t_{n+1}^{-1}t_{1}^{-1}, t_{n}^{-1}t_{1}^{-1},\ldots, t_{2}^{-1}t_{1}^{-1}, t_{1}^{-2}).
\end{align*}
Note that the image of $\varphi$ is the following subtorus in $T^{2n+2}$:
\begin{align*}
{\rm im}\ \varphi&=\{(1, t_{2}t_{1}^{-1},\ldots, t_{n+1}t_{1}^{-1}, t_{n+1}^{-1}t_{1}^{-1},\ t_{n}^{-1}t_{1}^{-1},\ldots, t_{2}^{-1}t_{1}^{-1},t_{1}^{-2})\ |\ (t_{1},\ldots, t_{n+1})\in T^{n+1}\} \\
&=\{(1, s_{1},\ldots, s_{n}, s_{n+1}, s_{n-1}^{-1}s_{n}s_{n+1},\ \ldots,\ s_{1}^{-1}s_{n}s_{n+1}, s_{n}s_{n+1})\ |\ (s_{1},\ldots, s_{n+1})\in T^{n+1}\}.
\end{align*}
By the fundamental theorem on homomorphisms, we have the following identifications:
\begin{align*}
{\rm im}\ \varphi\simeq T^{n+1}/\ker \varphi=T^{n+1}/\mathbb{Z}_{2}\simeq T^{n+1}.
\end{align*}
So the effective $T^{n+1}$-action is defined by the standard action of the subtorus ${\rm im} \ \varphi\subset T^{2n+2}$, i.e., for $[z_{1}:z_{2}:\cdots :z_{n+1}:z_{n+2}: z_{n+3}:\cdots :z_{2n+2}]\in Q_{2n}$ and $(s_{1},\ldots ,s_{n+1})\in T^{n+1}$, 
\begin{align}
\label{ef-ac}
&[z_{1}:z_{2}:\cdots :z_{n+1}:z_{n+2}: z_{n+3}:\cdots :z_{2n+2}] \\
&\mapsto 
[z_{1}:s_{1}z_{2}:\cdots :s_{n}z_{n+1}:s_{n+1}z_{n+2}: s_{n-1}^{-1}s_{n}s_{n+1}z_{n+3}:\cdots : s_{1}^{-1}s_{n}s_{n+1}z_{2n+1}: s_{n}s_{n+1}z_{2n+2}]. \nonumber
\end{align}
This action is nothing but the restricted $T^{n+1}$-action on $Q_{2n}$ from the standard $T^{2n+1}$-action on $\mathbb{C}P^{2n+2}$, where $T^{2n+1}=\{(1,t_{1},\ldots, t_{2n+1})\ |\ t_{1},\ldots, t_{2n+1}\in T^{1}\}\subset T^{2n+2}$.
Moreover, its axial function around $1\in V_{2n}=Q_{2n}^{T}$ coincides with \eqref{basis_vectors}.
Therefore, \eqref{basis_vectors} gives the effective $T^{n+1}$-action on $Q_{2n}$ by \eqref{ef-ac}.
\end{remark}


Applying a similar way to the other axial functions around each vertex (see \eqref{non-eff_axial_function}), 
we can define the axial function of the effective $T^{n+1}$-action 
as follows (see Figure~\ref{fig_0-cochain_presentation}).
\begin{definition}
\label{effective_axial_function}
Set $f:V_{2n}\to H^{2}(BT^{n+1})$ as 
\begin{align*}
f(j)=\left\{
\begin{array}{ll}
x_{j-1}-x_{n+1} & j=1,\ldots, n+2 \\
x_{n}-x_{2n+2-j} & j=n+3,\ldots, 2n+2
\end{array}
\right.
\end{align*}
where $x_{0}=0$ and $\langle x_{1},\ldots, x_{n+1}\rangle=H^{2}(BT^{n+1})$. 
Then we define the axial function $\alpha:E_{2n}\to H^{2}(BT^{n+1})$ as  
\begin{align*}
\alpha(ij):=f(j)-f(i)
\end{align*}
for $j\not=i, \overline{i}$.
\end{definition}

In this paper, 
the symbol $\mathcal{GQ}_{2n}$ represents 
the GKM graph $(\Gamma_{2n},\alpha)$ (or equivalently $(\Gamma_{2n},f)$, called a {\it $0$-cochain presentation}) for $\Gamma_{2n}=(V_{2n},E_{2n})$ defined in Definition~\ref{effective_axial_function}.

\begin{figure}[h]
\begin{tikzpicture}
\begin{scope}[xscale=0.4, yscale=0.4]

\fill(0,4) circle (5pt);
\node[above] at (0,4) {$-x_{3}$};
\fill(-4,2) circle (5pt);
\node[above] at (-4,2) {$x_{1}-x_{3}$};
\fill(-4,-2) circle (5pt);
\node[below] at (-4,-2) {$x_{2}-x_{3}$};
\fill(0,-4) circle (5pt);
\node[below] at (0,-4) {$x_{2}$};
\fill(4,-2) circle (5pt);
\node[below] at (4,-2) {$x_{2}-x_{1}$};
\fill(4,2) circle (5pt);
\node[above] at (4,2) {$0$};

\draw[thick] (0,4)--(-4,2);
\draw[thick] (0,4)--(-4,-2);
\draw[thick] (0,4)--(4,-2);
\draw[thick] (0,4)--(4,2);
\draw[thick] (-4,2)--(-4,-2);
\draw[thick] (-4,2)--(0,-4);
\draw[thick] (-4,2)--(4,2);
\draw[thick] (-4,-2)--(4,-2);
\draw[thick] (-4,-2)--(0,-4);
\draw[thick] (4,2)--(4,-2);
\draw[thick] (4,2)--(0,-4);
\draw[thick] (4,-2)--(0,-4);


\fill(12,4) circle (5pt);
\node[above] at (12,4) {$1$};
\fill(8,2) circle (5pt);
\node[left] at (8,2) {$2$};
\fill(8,-2) circle (5pt);
\node[left] at (8,-2) {$3$};
\fill(12,-4) circle (5pt);
\node[below] at (12,-4) {$6$};
\fill(16,-2) circle (5pt);
\node[right] at (16,-2) {$5$};
\fill(16,2) circle (5pt);
\node[right] at (16,2) {$4$};

\draw[thick] (12,4)--(8,2);
\draw[thick] (12,4)--(8,-2);
\draw[thick] (12,4)--(16,-2);
\draw[thick] (12,4)--(16,2);
\draw[thick] (8,2)--(8,-2);
\draw[thick] (8,2)--(12,-4);
\draw[thick] (8,2)--(16,2);
\draw[thick] (8,-2)--(16,-2);
\draw[thick] (8,-2)--(12,-4);
\draw[thick] (16,2)--(16,-2);
\draw[thick] (16,2)--(12,-4);
\draw[thick] (16,-2)--(12,-4);

\node[above] at (10,3) {$x_{1}$};
\draw[->] (12,4)--(10,3);
\node[above] at (14,3) {$x_{3}$};
\draw[->] (12,4)--(14,3);
\node[below] at (10,1) {$x_{2}$};
\draw[->] (12,4)--(10,1);
\node[below] at (14,1) {$x_{2}-x_{1}+x_{3}$};
\draw[->] (12,4)--(14,1);
\end{scope}
\end{tikzpicture}
\caption{The GKM graph $\mathcal{GQ}_{2n}$ when $n=2$ (also see the left graph in Figure~\ref{2-examples}). The right figure shows that the axial function $\alpha:E_{4}\to H^{2}(BT^{3})$ of $\mathcal{GQ}_{4}$ around the vertex $1$. The left figure shows its $0$-cochain presentation $f:V_{4}\to H^{2}(BT^{3})$.}
\label{fig_0-cochain_presentation}
\end{figure}
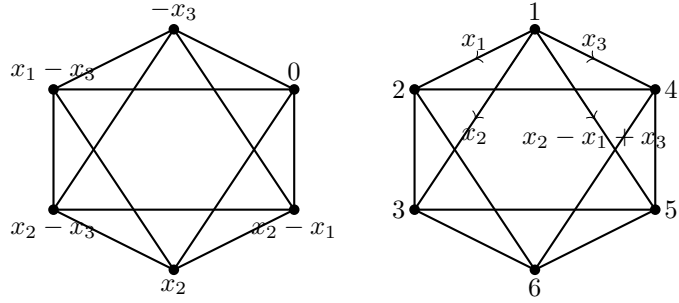

We exhibit some useful properties for the GKM graph $\mathcal{GQ}_{2n}$.
\begin{lemma}
\label{easy-equation1}
For every vertix $i\in V_{2n}$ in the GKM graph $\mathcal{GQ}_{2n}$, the following equation holds: 
\begin{align*}
f(i)+f(\overline{i})=x_{n}-x_{n+1}.
\end{align*}
\end{lemma}
\begin{proof}
Since $i+\overline{i}=2n+3$, we have that $x_{i-1}=x_{2n+2-\overline{i}}$.
The equality is immediately followed by Definition~\ref{effective_axial_function}.
\end{proof}

\begin{lemma}
\label{easy-equation2}
For every edge $ij\in E_{2n}$ in the GKM graph $\mathcal{GQ}_{2n}$, the following equation holds: 
\begin{align*}
\alpha(ij)=-\alpha(\overline{i}\overline{j}).
\end{align*}
\end{lemma}
\begin{proof}
By Lemma~\ref{easy-equation1}, we have 
\begin{align*}
\alpha(ij)&=f(j)-f(i)=(x_{n}-x_{n+1}-f(\overline{j}))-(x_{n}-x_{n+1}-f(\overline{i})) \\
&=f(\overline{i})-f(\overline{j})=\alpha(\overline{j}\overline{i})=-\alpha(\overline{i}\overline{j}).
\end{align*}
\end{proof}

\begin{lemma}
\label{well-definedness-of-M}
For every $j\in V_{2n}\setminus \{i,\overline{i}\}$, the following equation holds:
\begin{align*}
\alpha(ij)+\alpha(i\overline{j})=x_{n}-x_{n+1}-2f(i). 
\end{align*}
\end{lemma}
\begin{proof}
By definition of the axial function $\alpha:E_{2n}\to H^{2}(BT^{n+1})$ and Lemma~\ref{easy-equation1}, we have that 
\begin{align*}
\alpha(ij)+\alpha(i\overline{j})&=(f(j)-f(i))+(f(\overline{j})-f(i))=f(j)+f(\overline{j})-2f(i) \\
&=x_{n}-x_{n+1}-2f(i).
\end{align*} 
\end{proof}

\begin{lemma}
\label{three-independent}
The GKM graph $\mathcal{GQ}_{2n}$ is three-independent, i.e., for every vertex $i\in V_{2n}$ and every distinct three vertices $j_{1}, j_{2}, j_{3}\in V_{2n}\setminus \{i, \overline{i}\}$, the axial functions $\alpha(i j_{1}), \alpha(i j_{2}), \alpha(i j_{3})$ are linearly independent. 
\end{lemma}
\begin{proof}
This is straightforward from Definition~\ref{effective_axial_function}.
\end{proof}

\section{Two generators}
\label{sect:3}

The {\it graph equivariant cohomology} of the GKM graph $\mathcal{GQ}_{2n}$ is defined by 
\begin{align}
\label{graph-equivariant-cohomology}
H^{*}(\mathcal{GQ}_{2n}):=\{h:V_{2n}\to H^{*}(BT^{n+1})\ |\ h(i)-h(j)\equiv 0 \mod \alpha(ij)\ \text{for $ij\in E_{2n}$}\}.
\end{align}
The equation $h(i)-h(j)\equiv 0 \mod \alpha(ij)$ in \eqref{graph-equivariant-cohomology} is also called a {\it congruence relation}.
Note that $H^{*}(\mathcal{GQ}_{2n})$ has the graded $H^{*}(BT^{n+1})$-algebra structure induced by the graded algebra structure of $\bigoplus_{t\ge 0}R_{t}$, where $R_{t}$ is the degree $t$ part defined by $R_{t}:=H^{t}(BT^{n+1})$.
In particular, 
there is the injective homomorphism 
\begin{align}
\label{algebra-structure}
\iota:H^{*}(BT^{n+1}) \to H^{*}(\mathcal{GQ}_{2n})
\end{align}
such that the image of $x\in H^{*}(BT^{n+1})$, say $\iota(x):V_{2n}\to H^{*}(BT^{n+1})$,  is defined by the function 
\begin{align*}
\iota(x)(v)=x
\end{align*}
for all $v\in V_{2n}$.
This induces the $H^{*}(BT^{n+1})$-action on $H^{*}(\mathcal{GQ}_{2n})$.

The following lemma holds.
\begin{lemma}
\label{key-lemma}
For the effective $T^{n+1}$-action on $Q_{2n}$, the following graded $H^{*}(BT^{n+1})$-algebra isomorphism holds:
\begin{align*}
H_{T^{n+1}}^{*}(Q_{2n})\simeq H^{*}(\mathcal{GQ}_{2n}).
\end{align*}
\end{lemma}
\begin{proof}
Because the effective $T^{n+1}$-action on $Q_{2n}$ is obtained by the quotient $T^{n+1}/\mathbb{Z}_{2}$ by the finite kernel of the action \eqref{def-noneff-action}.
This implies that all isotropy subgroups of the effective $T^{n+1}$-action are connected. 
Therefore, by using $H^{odd}(Q_{2n})=0$ and \cite{FP07} (also see \cite[Theorem 2.12]{DKS}), we have the statement. 
\end{proof}
Lemma~\ref{key-lemma} means that to compute the equivariant cohomology $H_{T^{n+1}}^{*}(Q_{2n})$ is equivalent to compute the graph equivariant cohomology $H^{*}(\mathcal{GQ}_{2n})$.
The goal of this paper is to describe its generators and relations 
by the combinatorial data of the GKM graph $\mathcal{GQ}_{2n}$; this will be proved in Theorem~\ref{main-theorem2}.
The injective homomorphism \eqref{algebra-structure} for $H^{*}(\mathcal{GQ}_{2n})$ is also given in Proposition~\ref{polynomial generators}; toghether with Theorem~\ref{main-theorem2}, this establishes the $H^{*}(BT^{n+1})$-algebra structure on $H^{*}(\mathcal{GQ}_{2n})$.
To prove it, 
in this section, we introduce two types of elements in $H^{*}(\mathcal{GQ}_{2n})$ which will be the ring generators of $H^{*}(\mathcal{GQ}_{2n})$.

\subsection{Degree $2$ generators}
\label{sect:3.1}

We first define the degree two element, denoted by $M_{v}$, in $H^{2}(\mathcal{GQ}_{2n})$ for every $v\in V_{2n}$.

\begin{definition}[degree 2 generators]
\label{deg-2-gen}
Take a vertex $v\in V_{2n}=[2n+2]$. 
We define the function $M_{v}:V_{2n}\to H^{2}(BT^{n+1})$ by 
\begin{align*}
M_{v}(j)=\left\{
\begin{array}{ll}
0 & j=v \\
\alpha(jv)=f(v)-f(j) & j\not=v,\overline{v} \\
\alpha(\overline{v}k)+\alpha(\overline{v}\overline{k})=x_{n}-x_{n+1}-2f(\overline{v}) & j=\overline{v} 
\end{array}
\right.
\end{align*}
\end{definition}
The equality for $M_{v}(\overline{v})$ is obtained by Lemma~\ref{well-definedness-of-M}; this means that $M_{v}(\overline{v})$ does not depend on the choice of $k\in V_{2n}\setminus\{v,\overline{v}\}$.
The following proposition holds.
\begin{proposition}
For every $v\in V_{2n}$, the function $M_{v}:V_{2n}\to H^{2}(BT^{n+1})$ is an element of $H^{2}(\mathcal{GQ}_{2n})$, i.e., 
$M_{v}\in H^{2}(\mathcal{GQ}_{2n})$. 
\end{proposition}
\begin{proof}
We claim that $M_{v}(j)-M_{v}(k)\equiv 0 \mod \alpha(jk)$ for every $jk\in E_{2n}$ by case-by-case checking. 
\begin{description}
\item[The case when $j=v$]
For every $k\in V_{2n}\setminus \{j,\overline{j}\}= V_{2n}\setminus \{v,\overline{v}\}$, 
by Definition~\ref{deg-2-gen} we have 
\begin{align*}
M_{v}(j)-M_{v}(k)&=0-(f(v)-f(k))=f(k)-f(v)=\alpha(vk)\equiv 0 \mod \alpha(vk)=\alpha(jk).
\end{align*}
\item[The case when $j=\overline{v}$]
For every $k\in V_{2n}\setminus \{j,\overline{j}\}= V_{2n}\setminus \{v,\overline{v}\}$, 
we have that
\begin{align*}
M_{v}(j)-M_{v}(k)
&=\alpha(\overline{v}k)+\alpha(\overline{v}\overline{k})-\alpha(kv) \\
&=\alpha(\overline{v}k)+\alpha(\overline{v}\overline{k})-\alpha(\overline{v}\overline{k}) \quad ({\rm by\ Lemma~\ref{easy-equation2}\ and\ Definition~\ref{effective_axial_function}}) \\
&=\alpha(\overline{v}k)\equiv 0 \mod \alpha(\overline{v}k)=\alpha(jk).
\end{align*}
\item[The case when $j\not=v, \overline{v}$]
With the method similar to that demonstrated as above for the two cases ($k\not=\overline{v}$ and $k=\overline{v}$), we can easily check that $M_{v}(j)-M_{v}(k)\equiv 0\mod \alpha(jk)$.
\end{description}
Therefore, 
we have that $M_{v}\in H^{2}(\mathcal{GQ}_{2n})$.
\end{proof}

\begin{example}
For $n=2$, Figure~\ref{fig_deg2} represents the class $M_{6}\in H^{2}(\mathcal{GQ}_{4})$.
\begin{figure}[H]
\begin{tikzpicture}
\begin{scope}[xscale=0.4, yscale=0.4]
\fill(0,4) circle (5pt);
\node[above] at (0,4) {$M_{6}(1)=x_{2}-x_{3}-2f(\overline{6})=x_{2}+x_{3}$};
\fill(-4,2) circle (5pt);
\node[left] at (-4,2) {$M_{6}(2)=x_{2}-x_{1}+x_{3}$};
\fill(-4,-2) circle (5pt);
\node[left] at (-4,-2) {$M_{6}(3)=x_{3}$};
\fill(0,-4) circle (5pt);
\node[below] at (0,-4) {$M_{6}(6)=M_{6}(\overline{1})=0$};
\fill(4,-2) circle (5pt);
\node[right] at (4,-2) {$M_{6}(5)=M_{6}(\overline{2})=x_{1}$};
\fill(4,2) circle (5pt);
\node[right] at (4,2) {$M_{6}(4)=M_{6}(\overline{3})=x_{2}$};

\draw[very thick] (0,4)--(-4,2);
\draw[very thick] (0,4)--(-4,-2);
\draw[very thick] (0,4)--(4,-2);
\draw[very thick] (0,4)--(4,2);
\draw[very thick] (-4,2)--(-4,-2);
\draw[dashed] (-4,2)--(0,-4);
\draw[very thick] (-4,2)--(4,2);
\draw[very thick] (-4,-2)--(4,-2);
\draw[dashed] (-4,-2)--(0,-4);
\draw[very thick] (4,2)--(4,-2);
\draw[dashed] (4,2)--(0,-4);
\draw[dashed] (4,-2)--(0,-4);

\draw[->] (-4,2)--(-3,0.5);
\draw[->] (-4,-2)--(-3,-2.5);
\draw[->] (4,2)--(3,0.5);
\draw[->] (4,-2)--(3,-2.5);
\end{scope}
\end{tikzpicture}
\caption{The element $M_{6}\in H^{2}(\mathcal{GQ}_{4})$.}
\label{fig_deg2}
\end{figure}
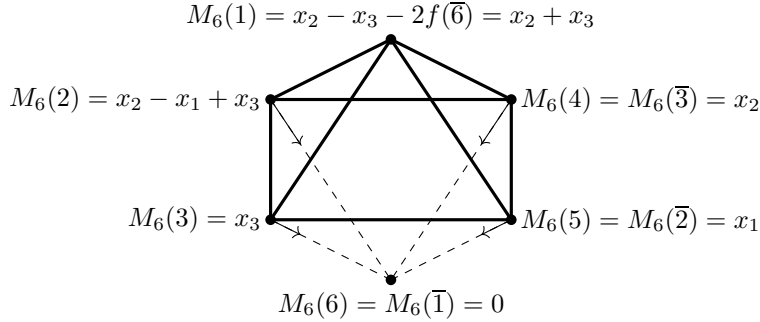
\end{example}

\subsection{Some properties for degree $2$ generators $M_{v}$}
\label{sect:3.2}
Before we define the higher degree generators, we introduce three properties for $M_{v}$'s.

For the vertices $W\subset V_{2n}$, we denote the {\it full-subgraph} with vertices $W$ by 
$\Gamma_{W}$, i.e., $\Gamma_{W}$ consists of the following data:
\begin{itemize}
\item the vertices $W$;
\item the edges $E_{W}:=\{ij\in E_{2n}\ |\ i,j\in W\}$.
\end{itemize}

Note that, by Definition~\ref{deg-2-gen}, 
the value of $M_{v}(j)\in H^{2}(BT^{n+1})$ for $j\not=\overline{v}$ coincides with the axial function $\alpha(jv)$ on the vertex $j$ of the full-subgraph $\Gamma_{I}$, where $I=V_{2n}\setminus \{v\}$. 
We also note that the edge $jv$ is the unique out-going edge of the full-subgraph $\Gamma_{I}$ from the vertex $j\in I\setminus\{\overline{v}\}$.
The following proposition shows that 
an element in $H^{2}(\mathcal{GQ}_{2n})$ with such a property is uniquely determined.

\begin{proposition}
If an element $A\in H^{2}(\mathcal{GQ}_{2n})$ satisfies that 
$A(j)=M_{v}(j)$ for every $j\in V_{2n}\setminus\{v,\overline{v}\}$, then $A=M_{v}$.
\end{proposition}
\begin{proof}
We first claim that $A(v)=0=M_{v}(v)$.
By using the congruence relations on the edges $jv$ for all $j\in V_{2n}\setminus\{v,\overline{v}\}$, 
we have 
\begin{align*}
A(j)-A(v)&=M_{v}(j)-A(v) 
=\alpha(jv)-A(v)\equiv -A(v) \equiv 0\mod \alpha(jv).
\end{align*}
This shows that for every $j\in V_{2n}\setminus\{v,\overline{v}\}$ 
there exists an integer $k_{j}$ such that 
\begin{align*}
A(v)=-k_{j}\alpha(jv)=k_{j}\alpha(vj). 
\end{align*}
In particular, for every $j_{1},\ j_{2} \in V_{2n}\setminus \{v,\overline{v}\}$, 
\begin{align*}
A(v)=k_{j_{1}}\alpha(vj_{1})=k_{j_{2}}\alpha(vj_{2}). 
\end{align*}
By Lemma~\ref{three-independent},  
this gives that $k_{j}=0$, 
thus establishing $A(v)=0$.

We next claim that $A(\overline{v})=x_{n}-x_{n+1}-2f(\overline{v})=M_{v}(\overline{v})$.
By using the congruence relations on the edges $j\overline{v}$ for all $j\in V_{2n}\setminus \{v,\overline{v}\}$, we have 
\begin{align*}
A(j)-A(\overline{v})&=M_{v}(j)-A(\overline{v}) 
=\alpha(jv)-A(\overline{v}) \equiv 0\mod \alpha(j\overline{v}).
\end{align*}
This shows that for every $j\in V_{2n}\setminus \{v,\overline{v}\}$ there exists an integer $k_{j}$ which satisfies the following equation:
\begin{align*}
A(\overline{v})&=\alpha(jv)+k_{j}\alpha(j\overline{v}) \\
&=\alpha(\overline{v}\overline{j})+k_{j}\alpha(v\overline{j}) \quad ({\rm by\ Lemma~\ref{easy-equation2}\ and\ Definition~\ref{effective_axial_function}}) \\
&=x_{n}-x_{n+1}-2f(\overline{v})-\alpha(\overline{v}j)+k_{j}\alpha(v\overline{j}) \quad ({\rm by\ Lemma~\ref{well-definedness-of-M}}) \\
&=x_{n}-x_{n+1}-2f(\overline{v})-(1+k_{j})\alpha(\overline{v}j) \quad ({\rm by\ Lemma~\ref{easy-equation2}})
\end{align*}
In particular, this equation holds for every $j_{1},\ j_{2} \in V_{2n}\setminus \{v,\overline{v}\}$.
Therefore, 
by using the similar method for the proof of $A(v)=0$, we obtain $1+k_{j}=0$, thus $k_{j}=-1$.
Therefore, by Lemma~\ref{easy-equation2} and Lemma~\ref{well-definedness-of-M},
\begin{align*}
A(\overline{v})=\alpha(jv)-\alpha(j\overline{v})=x_{n}-x_{n+1}-2f(\overline{v}).
\end{align*}
This establishes $A=M_{v}$.
\end{proof}

Recall $\iota:H^{*}(BT^{n+1})\to H^{*}(\mathcal{GQ}_{2n})$ in \eqref{algebra-structure}. 
By abuse of notation, we also denote $\iota(x):V_{2n}\to H^{*}(BT^{n+1})$ as $x:V_{2n}\to H^{*}(BT^{n+1})$ for an element $x\in H^{*}(BT^{n+1})$.
The following proposition shows that $x$ can be also presented by $M_{v}$'s. 
\begin{proposition}
\label{polynomial generators}
The generator $x_{i}\in H^{*}(BT^{n+1})$ for $i=1,\ldots, n+1$ is obtained by the following equality: 
\begin{align*}
x_{i}=M_{i+1}-M_{1}.
\end{align*}
\end{proposition}
\begin{proof}
Because $i=1,\ldots, n+1$, for all $j\in V_{2n}\setminus\{\overline{1}, \overline{i+1}\}$, we have that 
\begin{align*}
M_{i+1}(j)-M_{1}(j)&=f(i+1)-f(j)-(f(1)-f(j))=f(i+1)-f(1)=x_{i}-x_{n+1}-(x_{0}-x_{n+1})=x_{i}.
\end{align*}
For $j=\overline{1}=2n+2$, we have 
\begin{align*}
M_{i+1}(2n+2)-M_{1}(2n+2)&=f(i+1)-f(2n+2)-(x_{n}-x_{n+1}-2f(2n+2)) \\
&=f(i+1)+f(2n+2)-(f(i+1)+f(\overline{i+1}))\quad ({\rm by\ Lemma~\ref{easy-equation1}}) \\
&=f(2n+2)-f(\overline{i+1})=x_{i} \quad ({\rm by\ Definition~\ref{effective_axial_function}}).
\end{align*}
For $j=\overline{i+1}=2n+2-i$, we have 
\begin{align*}
M_{i+1}(2n+2-i)-M_{1}(2n+2-i)&=(x_{n}-x_{n+1}-2f(2n+2-i))-(f(1)-f(2n+2-i)) \\
&=x_{n}-f(2n+2-i) \quad ({\rm by\ Definition~\ref{effective_axial_function}}).
\end{align*}
In this case, by using Definition~\ref{effective_axial_function} again, we have the following equations.
\begin{align*}
x_{n}-f(2n+2-i) =
\left\{
\begin{array}{ll}
x_{n}-(x_{n}-x_{i}) & i=1,\ldots, n-1 \\
x_{n}-(x_{2n+1-i}-x_{n+1}) & i=n,n+1 
\end{array}
\right.
\end{align*}
Therefore, $M_{i+1}(2n+2-i)-M_{1}(2n+2-i)=x_{i}$.
These equations show that 
$M_{i+1}(v)-M_{1}(v)=x_{i}$
for all $v\in V_{2n}$.
This establishes the statement. 
\end{proof}

We also have the following proposition for the $0$-cochain presentation $f:V_{2n}\to H^{2}(BT^{n+1})$ defined in Definition~\ref{effective_axial_function}.
\begin{proposition}
\label{rel_M=-f}
The $0$-cochain presentaion $f:V_{2n}\to H^{2}(BT^{n+1})$ satisfies that 
$f=-M_{n+2}$.
\end{proposition}
\begin{proof}
By definitions of $f$ and $M_{n+2}$, we can easily check the statement. 
\end{proof}

\begin{example}
\label{using_ex_Fig3}
The left figure of Figure~\ref{fig_0-cochain_presentation} in Section~\ref{sect:2.2} also represents that $f=-M_{4}$. 
\end{example}

\subsection{Higher degree generators}
\label{sect:3.3}

We next define the degree $2l$ element $\Delta_{K}$ in $H^{2l}(\mathcal{GQ}_{2n})$ for some $K\subset V_{2n}$ such that $|K|=l+1$, where $|K|$ is the cardinality of $K$.

For a non-empty subset $K\subset V_{2n}$, by definition of $\Gamma_{2n}$,
the following two properties are equivalent:
\begin{itemize}
\item the full-subgraph $\Gamma_{K}$ is  the complete subgraph of $\Gamma_{2n}$;
\item if $i\in K$, then $\overline{i}\not\in K$ (or equivalently $\{i, \overline{i}\}\not\subset K$ for all $i\in V_{2n}$).
\end{itemize}
We call one of these properties {\it the property $(\ast)$}.
Note that if $K$ satisfies the property $(*)$, then its cardinality satisfies $1\le |K|\le n+1$.

\begin{definition}[degree $(\ge )2n$ generators]
\label{higher_deg_generators}
Let $K\subset V_{2n}=[2n+2]$ be a non-empty subset that satisfies the property $(\ast)$.
We define the function $\Delta_{K}:V_{2n}\to H^{4n-2|K|+2}(BT^{n+1})$ by the following map.
\begin{align}
\label{higher-generator}
\Delta_{K}(j)=\left\{
\begin{array}{ll}
{\displaystyle \prod_{k\not\in K\cup \{\overline{j}\}}\alpha(jk)=\prod_{k\not\in K\cup \{\overline{j}\}}(f(k)-f(j))} & j\in K \\
0 & j\not\in K
\end{array}
\right.
\end{align}
\end{definition}
Note that $\Delta_{K}$ is nothing but the {\it Thom class} of the GKM subgraph $\Gamma_{K}$ (see \cite[Section 4]{MMP}). 
Therefore, by the similar arguments for the proof of \cite[Lemma 4.1]{MMP}, we have the following lemma.
\begin{lemma}  
If $K\subset V_{2n}$ satisfies the property $(\ast)$, then  
$\Delta_{K}\in H^{4n-2|K|+2}(\mathcal{GQ}_{2n})$.
\end{lemma}
\begin{remark}
\label{about_Thom_class}
Geometrically, $\Delta_{K}$ is the equivariant Thom class of the invariant submanifold in $Q_{2n}$ (see \cite{Ma99}) which is diffeomorphic to the projective space whose fixed points consisting of $K$.
For example, there exists the following subspace which is diffeomorphic to $\mathbb{C}P^{l-1}$ in $Q_{2n}$ for every $1\le l\le n+1$:
\begin{align*}
\{[z_{1}:z_{2}:\cdots :z_{l}:0:\cdots:0]\in Q_{2n}\ |\ z_{i}\in \mathbb{C}\}\simeq \mathbb{C}P^{l-1}.
\end{align*}
In this case, $K=[l]\subset [2n+2]$ and 
the class $\Delta_{K}\in H^{4n-2l+2}(\mathcal{GQ}_{2n})$ corresponds to the equivariant Thom class of $\mathbb{C}P^{l-1}\subset Q_{2n}$ in the equivariant cohomology $H_{T}^{4n-2l+2}(Q_{2n})$.
\end{remark}

\begin{example}
For the GKM graph $\mathcal{GQ}_{4}$, the set of vertices $K=\{1,2,3\}$ satisfies the property $(\ast)$.
Figure~\ref{fig_higher_gen} represents the class $\Delta_{K}\in H^{4}(\mathcal{GQ}_{4})$. 
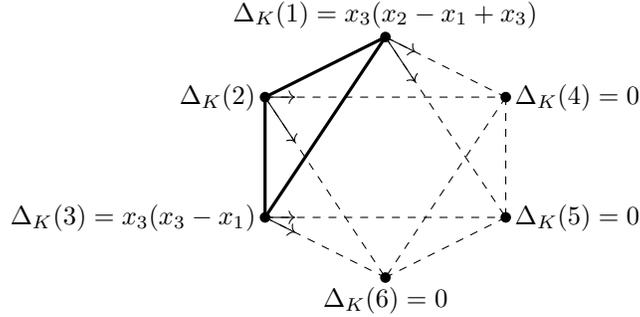
\begin{figure}[H]
\begin{tikzpicture}
\begin{scope}[xscale=0.4, yscale=0.4]
\fill(0,4) circle (5pt);
\node[above] at (0,4) {$\Delta_{K}(1)=x_{3}(x_{2}-x_{1}+x_{3})$};
\fill(-4,2) circle (5pt);
\node[left] at (-4,2) {$\Delta_{K}(2)$};
\fill(-4,-2) circle (5pt);
\node[left] at (-4,-2) {$\Delta_{K}(3)=x_{3}(x_{3}-x_{1})$};
\fill(0,-4) circle (5pt);
\node[below] at (0,-4) {$\Delta_{K}(6)=0$};
\fill(4,-2) circle (5pt);
\node[right] at (4,-2) {$\Delta_{K}(5)=0$};
\fill(4,2) circle (5pt);
\node[right] at (4,2) {$\Delta_{K}(4)=0$};

\draw[very thick] (0,4)--(-4,2);
\draw[very thick] (0,4)--(-4,-2);
\draw[dashed] (0,4)--(4,-2);
\draw[dashed] (0,4)--(4,2);
\draw[very thick] (-4,2)--(-4,-2);
\draw[dashed] (-4,2)--(0,-4);
\draw[dashed] (-4,2)--(4,2);
\draw[dashed] (-4,-2)--(4,-2);
\draw[dashed] (-4,-2)--(0,-4);
\draw[dashed] (4,2)--(4,-2);
\draw[dashed] (4,2)--(0,-4);
\draw[dashed] (4,-2)--(0,-4);

\draw[->] (0,4)--(1,3.5);
\draw[->] (0,4)--(1,2.5);

\draw[->] (-4,2)--(-3,0.5);
\draw[->] (-4,2)--(-3,2);

\draw[->] (-4,-2)--(-3,-2.5);
\draw[->] (-4,-2)--(-3,-2);
\end{scope}
\end{tikzpicture}
\caption{$\Delta_{K}$ for $K=\{1,2,3\}$, where $\Delta_{K}(2)=(x_{3}-x_{1})(x_{2}-x_{1}+x_{3})$.}
\label{fig_higher_gen}
\end{figure}
\end{example}

\begin{example}
For the GKM graph $\mathcal{GQ}_{4}$, the set of vertices $L=\{1,2\}$ also satisfies the property $(*)$.
Figure~\ref{fig_higher_gen2} represents the class $\Delta_{L}\in H^{6}(\mathcal{GQ}_{4})$. 
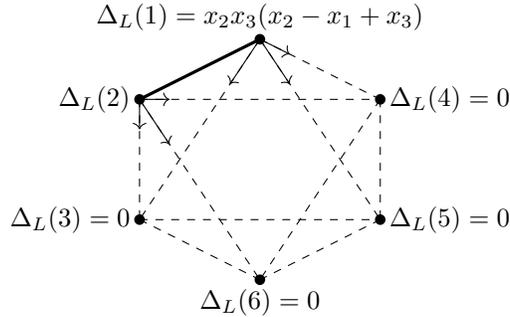
\begin{figure}[H]
\begin{tikzpicture}
\begin{scope}[xscale=0.4, yscale=0.4]
\fill(0,4) circle (5pt);
\node[above] at (0,4) {$\Delta_{L}(1)=x_{2}x_{3}(x_{2}-x_{1}+x_{3})$};
\fill(-4,2) circle (5pt);
\node[left] at (-4,2) {$\Delta_{L}(2)$};
\fill(-4,-2) circle (5pt);
\node[left] at (-4,-2) {$\Delta_{L}(3)=0$};
\fill(0,-4) circle (5pt);
\node[below] at (0,-4) {$\Delta_{L}(6)=0$};
\fill(4,-2) circle (5pt);
\node[right] at (4,-2) {$\Delta_{L}(5)=0$};
\fill(4,2) circle (5pt);
\node[right] at (4,2) {$\Delta_{L}(4)=0$};

\draw[very thick] (0,4)--(-4,2);
\draw[dashed] (0,4)--(-4,-2);
\draw[dashed] (0,4)--(4,-2);
\draw[dashed] (0,4)--(4,2);
\draw[dashed] (-4,2)--(-4,-2);
\draw[dashed] (-4,2)--(0,-4);
\draw[dashed] (-4,2)--(4,2);
\draw[dashed] (-4,-2)--(4,-2);
\draw[dashed] (-4,-2)--(0,-4);
\draw[dashed] (4,2)--(4,-2);
\draw[dashed] (4,2)--(0,-4);
\draw[dashed] (4,-2)--(0,-4);

\draw[->] (0,4)--(1,3.5);
\draw[->] (0,4)--(1,2.5);
\draw[->] (0,4)--(-1,2.5);

\draw[->] (-4,2)--(-3,0.5);
\draw[->] (-4,2)--(-3,2);
\draw[->] (-4,2)--(-4,1);

\end{scope}
\end{tikzpicture}
\caption{$\Delta_{L}$ for $L=\{1,2\}$, where $\Delta_{L}(2)=(x_{2}-x_{1})(x_{3}-x_{1})(x_{2}-x_{1}+x_{3})$}
\label{fig_higher_gen2}
\end{figure}
\end{example}

\section{Four relations}
\label{sect:4}

In this section, we introduce the four types of relations among $M_{v}$'s and $\Delta_{K}$'s (see Lemma~\ref{lem-for-rel1}--~\ref{lem-for-rel4}).

\begin{relation}[$\prod_{\cap J=\emptyset}G_{J}=0$]
\label{relation1}
We use the following notation for $J\subset V_{2n}$.
\begin{align}
\label{def_G}
G_{J}:=\left\{
\begin{array}{ll}
M_{v} & \text{if $J=V_{2n}\setminus\{v\}$ for a vertex $v\in V_{2n}$} \\
\Delta_{J} & \text{if $J$ satisfies that the property $(\ast)$, i.e., $\{i,\overline{i}\}\not\subset J$ for every $i\in V_{2n}$}
\end{array}
\right.
\end{align}
By Definition~\ref{deg-2-gen} and Definition~\ref{higher_deg_generators},
we have the following relation in $H^{*}(\mathcal{GQ}_{2n})$ (see Figure~\ref{fig_rel1}).
\begin{lemma}[Relation 1]
\label{lem-for-rel1}
There is the following relation:
\begin{align}
\label{1st-rel}
\prod_{\cap J=\emptyset}G_{J}=0.
\end{align}
\end{lemma}
\end{relation}

\begin{figure}[H]
\begin{tikzpicture}
\begin{scope}[xscale=0.3, yscale=0.3]
\fill(0,4) circle (7pt);
\node[above] at (0,4) {$\Delta_{\{1\}}(1)$}; 
\fill(-4,2) circle (5pt);
\node[left] at (-4,2) {$0$};
\fill(-4,-2) circle (5pt);
\node[left] at (-4,-2) {$0$};
\fill(0,-4) circle (5pt);
\node[below] at (0,-4) {$0$};
\fill(4,-2) circle (5pt);
\node[right] at (4,-2) {$0$};
\fill(4,2) circle (5pt);
\node[right] at (4,2) {$0$};

\draw[dashed] (0,4)--(-4,2);
\draw[dashed] (0,4)--(-4,-2);
\draw[dashed] (0,4)--(4,-2);
\draw[dashed] (0,4)--(4,2);
\draw[dashed] (-4,2)--(-4,-2);
\draw[dashed] (-4,2)--(0,-4);
\draw[dashed] (-4,2)--(4,2);
\draw[dashed] (-4,-2)--(4,-2);
\draw[dashed] (-4,-2)--(0,-4);
\draw[dashed] (4,2)--(4,-2);
\draw[dashed] (4,2)--(0,-4);
\draw[dashed] (4,-2)--(0,-4);

\draw[->] (0,4)--(1,3.5);
\draw[->] (0,4)--(1,2.5);
\draw[->] (0,4)--(-1,2.5);
\draw[->] (0,4)--(-1,3.5);

\node at (6,0) {$\times$};

\fill(12,4) circle (5pt);
\node[above] at (12,4) {$0$};
\fill(8,2) circle (5pt);
\node[above] at (8,2) {$M_{1}(2)$};
\fill(8,-2) circle (5pt);
\node[below] at (8,-2) {$M_{1}(3)$};
\fill(12,-4) circle (5pt);
\node[below] at (12,-4) {$M_{1}(6)$};
\fill(16,-2) circle (5pt);
\node[below] at (16,-2) {$M_{1}(5)$};
\fill(16,2) circle (5pt);
\node[above] at (16,2) {$M_{1}(4)$};

\draw[dashed] (12,4)--(8,2);
\draw[dashed] (12,4)--(8,-2);
\draw[dashed] (12,4)--(16,-2);
\draw[dashed] (12,4)--(16,2);
\draw[very thick] (8,2)--(8,-2);
\draw[very thick] (8,2)--(12,-4);
\draw[very thick] (8,2)--(16,2);
\draw[very thick] (8,-2)--(16,-2);
\draw[very thick] (8,-2)--(12,-4);
\draw[very thick] (16,2)--(16,-2);
\draw[very thick] (16,2)--(12,-4);
\draw[very thick] (16,-2)--(12,-4);

\draw[->] (8,2)--(10,3);
\draw[->] (8,-2)--(10,1);
\draw[->] (16,-2)--(14,1);
\draw[->] (16,2)--(14,3);

\node at (18,0) {$=$};

\fill(24,4) circle (5pt);
\node[above] at (24,4) {$0$};
\fill(20,2) circle (5pt);
\node[above] at (20,2) {$0$};
\fill(20,-2) circle (5pt);
\node[below] at (20,-2) {$0$};
\fill(24,-4) circle (5pt);
\node[below] at (24,-4) {$0$};
\fill(28,-2) circle (5pt);
\node[below] at (28,-2) {$0$};
\fill(28,2) circle (5pt);
\node[above] at (28,2) {$0$};

\draw[thick] (24,4)--(20,2);
\draw[thick] (24,4)--(20,-2);
\draw[thick] (24,4)--(28,-2);
\draw[thick] (24,4)--(28,2);
\draw[thick] (20,2)--(20,-2);
\draw[thick] (20,2)--(24,-4);
\draw[thick] (20,2)--(28,2);
\draw[thick] (20,-2)--(28,-2);
\draw[thick] (20,-2)--(24,-4);
\draw[thick] (28,2)--(28,-2);
\draw[thick] (28,2)--(24,-4);
\draw[thick] (28,-2)--(24,-4);
\end{scope}
\end{tikzpicture}
\caption{In $H^{*}(\mathcal{GQ}_{4})$, the relation $\Delta_{\{1\}}\cdot M_{1}=0$ holds  (Relation~\ref{relation1}).}
\label{fig_rel1}
\end{figure}
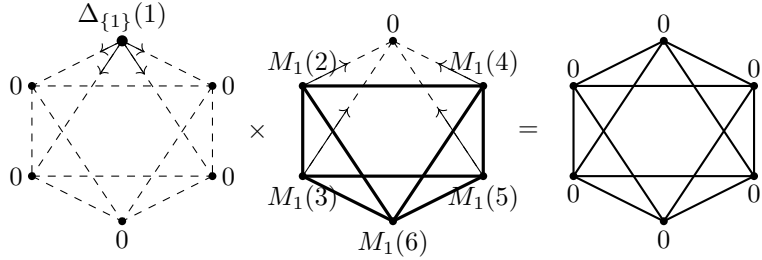

\begin{relation}[$M_{i}+M_{\overline{i}}=M_{j}+M_{\overline{j}}$]
\label{relation2}
We define the element $X\in H^{2}(\mathcal{GQ}_{2n})$ by the map $X:V_{2n}\to H^{2}(BT^{n+1})$ which satisfies  
\begin{align*}
X(k):=x_{n}-x_{n+1}-2f(k)
\end{align*}
for all $k\in V_{2n}$.
Then, by Lemma~\ref{well-definedness-of-M}, 
for every $j\in V_{2n}$ such that $j, \overline{j}\not=k$, there exists the following equation: 
\begin{align*}
X(k)=
\alpha(kj)+\alpha(k\overline{j}).
\end{align*}
By Lemma~\ref{easy-equation1} and 
Definition~\ref{deg-2-gen}, we have the following relation (see Figure~\ref{fig_rel2}).
\begin{lemma}[Relation 2]
\label{lem-for-rel2}
For every $v\in V_{2n}$, 
there is the following relation:
\begin{align}
\label{2nd-rel}
M_{v}+M_{\overline{v}}=X.
\end{align}
\end{lemma}
\end{relation}

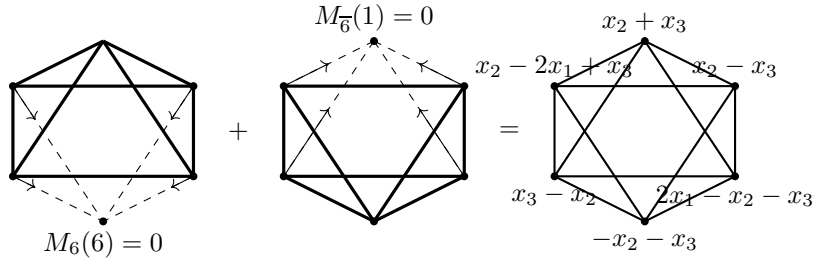
\begin{figure}[H]
\begin{tikzpicture}
\begin{scope}[xscale=0.3, yscale=0.3]
\fill(-4,2) circle (5pt);
\fill(-4,-2) circle (5pt);
\fill(0,-4) circle (5pt);
\fill(4,-2) circle (5pt);
\fill(4,2) circle (5pt);

\draw[very thick] (0,4)--(-4,2);
\draw[very thick] (0,4)--(-4,-2);
\draw[very thick] (0,4)--(4,-2);
\draw[very thick] (0,4)--(4,2);
\draw[very thick] (-4,2)--(-4,-2);
\draw[dashed] (-4,2)--(0,-4);
\draw[very thick] (-4,2)--(4,2);
\draw[very thick] (-4,-2)--(4,-2);
\draw[dashed] (-4,-2)--(0,-4);
\draw[very thick] (4,2)--(4,-2);
\draw[dashed] (4,2)--(0,-4);
\draw[dashed] (4,-2)--(0,-4);

\draw[->] (-4,2)--(-3,0.5);
\draw[->] (-4,-2)--(-3,-2.5);
\draw[->] (4,2)--(3,0.5);
\draw[->] (4,-2)--(3,-2.5);

\node at (6,0) {$+$};

\fill(12,4) circle (5pt);
\fill(8,2) circle (5pt);
\fill(8,-2) circle (5pt);
\fill(12,-4) circle (5pt);
\fill(16,-2) circle (5pt);
\fill(16,2) circle (5pt);

\draw[dashed] (12,4)--(8,2);
\draw[dashed] (12,4)--(8,-2);
\draw[dashed] (12,4)--(16,-2);
\draw[dashed] (12,4)--(16,2);
\draw[very thick] (8,2)--(8,-2);
\draw[very thick] (8,2)--(12,-4);
\draw[very thick] (8,2)--(16,2);
\draw[very thick] (8,-2)--(16,-2);
\draw[very thick] (8,-2)--(12,-4);
\draw[very thick] (16,2)--(16,-2);
\draw[very thick] (16,2)--(12,-4);
\draw[very thick] (16,-2)--(12,-4);

\draw[->] (8,2)--(10,3);
\draw[->] (8,-2)--(10,1);
\draw[->] (16,-2)--(14,1);
\draw[->] (16,2)--(14,3);

\node at (18,0) {$=$};

\fill(24,4) circle (5pt);
\node[above] at (24,4) {$x_{2}+x_{3}$};
\fill(20,2) circle (5pt);
\node[above] at (20,2) {$x_{2}-2x_{1}+x_{3}$};
\fill(20,-2) circle (5pt);
\node[below] at (20,-2) {$x_{3}-x_{2}$};
\fill(24,-4) circle (5pt);
\node[below] at (24,-4) {$-x_{2}-x_{3}$};
\fill(28,-2) circle (5pt);
\node[below] at (28,-2) {$2x_{1}-x_{2}-x_{3}$};
\fill(28,2) circle (5pt);
\node[above] at (28,2) {$x_{2}-x_{3}$};

\draw[thick] (24,4)--(20,2);
\draw[thick] (24,4)--(20,-2);
\draw[thick] (24,4)--(28,-2);
\draw[thick] (24,4)--(28,2);
\draw[thick] (20,2)--(20,-2);
\draw[thick] (20,2)--(24,-4);
\draw[thick] (20,2)--(28,2);
\draw[thick] (20,-2)--(28,-2);
\draw[thick] (20,-2)--(24,-4);
\draw[thick] (28,2)--(28,-2);
\draw[thick] (28,2)--(24,-4);
\draw[thick] (28,-2)--(24,-4);

\node[above] at (12,4) {$M_{\overline{6}}(1)=0$};
\node[below] at (0,-4) {$M_{6}(6)=0$};

\end{scope}
\end{tikzpicture}
\caption{$M_{6}+M_{\overline{6}}=X$ (Relation~\ref{relation2}), where 
$\overline{6}=1\in V_{4}$.}
\label{fig_rel2}
\end{figure}

\begin{relation}[$\prod_{i\in I}M_{i}=\Delta_{(I\cup \{a\})^{c}}+\Delta_{(I\cup \{\overline{a}\})^{c}}$]
\label{relation3}
Assume that the subset $I\subset V_{2n}$ satisfies that $|I|=n$ and the property $(\ast)$. 
Then, because $|V_{2n}|=2n+2$, there exists the unique pair $\{a,\overline{a}\}\subset I^{c}=V_{2n}\setminus I$ such that
\begin{align*}
\Delta_{K},\quad \Delta_{L}\in H^{2n}(\mathcal{GQ}_{2n})
\end{align*}
for 
$K=(I\cup \{a\})^{c}=I^{c}\setminus \{a\}$ and $L=(I\cup \{\overline{a}\})^{c}=I^{c}\setminus \{\overline{a}\}$.
Then, the following relation holds (see Figure~\ref{fig_rel3} for $n=2$ and Figure~\ref{fig_rel3-ver2} for $n=3$).
\begin{lemma}[Relation 3]
\label{lem-for-rel3}
For every $I\subset V_{2n}$ as above, 
there is the following relation:
\begin{align}
\label{3rd-rel}
\prod_{i\in I} M_{i}=\Delta_{(I\cup \{a\})^{c}}+\Delta_{(I\cup \{\overline{a}\})^{c}}.
\end{align} 
\end{lemma}
\begin{proof}
If $v\in I$, then  
$\prod_{i\in I} M_{i}(v)=0=\Delta_{(I\cup \{a\})^{c}}(v)+\Delta_{(I\cup \{\overline{a}\})^{c}}(v).$
If $v\in I^{c}\setminus \{a,\overline{a}\}$, 
then $\overline{v}\in I$.
Thus, we have
\begin{align*}
\prod_{i\in I} M_{i}(v)=
M_{\overline{v}}(v)\prod_{i\in I\setminus\{\overline{v}\}}\alpha(vi)
&=(\alpha(va)+\alpha(v\overline{a}))\prod_{i\in I\setminus\{\overline{v}\}}\alpha(vi) \\
&=\Delta_{(I\cup \{a\})^{c}}(v)+\Delta_{(I\cup \{\overline{a}\})^{c}}(v).
\end{align*}
For the vertex $a(\not\in I)$, we have 
\begin{align*}
\Delta_{(I\cup \{a\})^{c}}(a)+\Delta_{(I\cup \{\overline{a}\})^{c}}(a)=0+\prod_{i\in I}\alpha(ai)=\prod_{i\in I} M_{i}(a). 
\end{align*}
Similarly, we have $\prod_{i\in I} M_{i}(\overline{a})=\Delta_{(I\cup \{a\})^{c}}(\overline{a})+\Delta_{(I\cup \{\overline{a}\})^{c}}(\overline{a})$.
This establishes the statement.
\end{proof}
\end{relation}

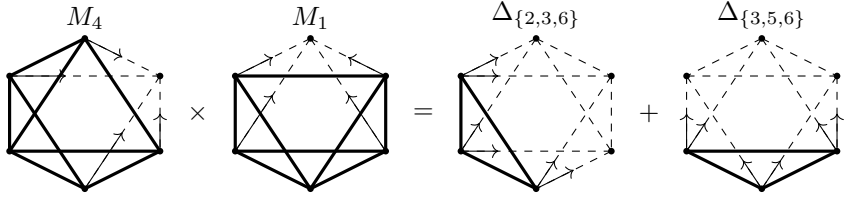
\begin{figure}[H]
\begin{tikzpicture}
\begin{scope}[xscale=0.25, yscale=0.25]
\fill(0,4) circle (5pt);
\fill(-4,2) circle (5pt);
\fill(-4,-2) circle (5pt);
\fill(0,-4) circle (5pt);
\fill(4,-2) circle (5pt);
\fill(4,2) circle (5pt);

\draw[very thick] (0,4)--(-4,2);
\draw[very thick] (0,4)--(-4,-2);
\draw[very thick] (0,4)--(4,-2);
\draw[dashed] (0,4)--(4,2);
\draw[very thick] (-4,2)--(-4,-2);
\draw[very thick] (-4,2)--(0,-4);
\draw[dashed] (-4,2)--(4,2);
\draw[very thick] (-4,-2)--(4,-2);
\draw[very thick] (-4,-2)--(0,-4);
\draw[dashed] (4,2)--(4,-2);
\draw[dashed] (4,2)--(0,-4);
\draw[very thick] (4,-2)--(0,-4);

\draw[->] (-4,2)--(-1,2);
\draw[->] (0,4)--(2,3);
\draw[->] (4,-2)--(4,0);
\draw[->] (0,-4)--(2,-1);

\node at (6,0) {$\times$};

\fill(12,4) circle (5pt);
\fill(8,2) circle (5pt);
\fill(8,-2) circle (5pt);
\fill(12,-4) circle (5pt);
\fill(16,-2) circle (5pt);
\fill(16,2) circle (5pt);

\draw[dashed] (12,4)--(8,2);
\draw[dashed] (12,4)--(8,-2);
\draw[dashed] (12,4)--(16,-2);
\draw[dashed] (12,4)--(16,2);
\draw[very thick] (8,2)--(8,-2);
\draw[very thick] (8,2)--(12,-4);
\draw[very thick] (8,2)--(16,2);
\draw[very thick] (8,-2)--(16,-2);
\draw[very thick] (8,-2)--(12,-4);
\draw[very thick] (16,2)--(16,-2);
\draw[very thick] (16,2)--(12,-4);
\draw[very thick] (16,-2)--(12,-4);

\draw[->] (8,2)--(10,3);
\draw[->] (8,-2)--(10,1);
\draw[->] (16,-2)--(14,1);
\draw[->] (16,2)--(14,3);

\node[above] at (12,4) {$M_{1}$};
\node[above] at (0,4) {$M_{4}$};

\node at (18,0) {$=$};


\fill(24,4) circle (5pt);
\fill(20,2) circle (5pt);
\fill(20,-2) circle (5pt);
\fill(24,-4) circle (5pt);
\fill(28,-2) circle (5pt);
\fill(28,2) circle (5pt);

\draw[dashed] (24,4)--(20,2);
\draw[dashed] (24,4)--(20,-2);
\draw[dashed] (24,4)--(28,-2);
\draw[dashed] (24,4)--(28,2);
\draw[very thick] (20,2)--(20,-2);
\draw[very thick] (20,2)--(24,-4);
\draw[dashed] (20,2)--(28,2);
\draw[dashed] (20,-2)--(28,-2);
\draw[very thick] (20,-2)--(24,-4);
\draw[dashed] (28,2)--(28,-2);
\draw[dashed] (28,2)--(24,-4);
\draw[dashed] (28,-2)--(24,-4);

\draw[->] (20,2)--(22,2);
\draw[->] (20,2)--(22,3);
\draw[->] (20,-2)--(22,-2);
\draw[->] (20,-2)--(21,-0.5);
\draw[->] (24,-4)--(25,-2.5);
\draw[->] (24,-4)--(26,-3);

\node[above] at (24,4) {$\Delta_{\{2,3,6\}}$};

\node at (30,0) {$+$};

\fill(36,4) circle (5pt);
\fill(32,2) circle (5pt);
\fill(32,-2) circle (5pt);
\fill(36,-4) circle (5pt);
\fill(40,-2) circle (5pt);
\fill(40,2) circle (5pt);

\draw[dashed] (36,4)--(32,2);
\draw[dashed] (36,4)--(32,-2);
\draw[dashed] (36,4)--(40,-2);
\draw[dashed] (36,4)--(40,2);
\draw[dashed] (32,2)--(32,-2);
\draw[dashed] (32,2)--(36,-4);
\draw[dashed] (32,2)--(40,2);
\draw[very thick] (32,-2)--(40,-2);
\draw[very thick] (32,-2)--(36,-4);
\draw[dashed] (40,2)--(40,-2);
\draw[dashed] (40,2)--(36,-4);
\draw[very thick] (40,-2)--(36,-4);

\draw[->] (32,-2)--(32,0);
\draw[->] (32,-2)--(33,-0.5);
\draw[->] (40,-2)--(39,-0.5);
\draw[->] (40,-2)--(40,0);
\draw[->] (36,-4)--(35,-2.5);
\draw[->] (36,-4)--(37,-2.5);

\node[above] at (36,4) {$\Delta_{\{3,5,6\}}$};
\end{scope}
\end{tikzpicture}
\caption{For the GKM graph $\mathcal{GQ}_{4}$ (where the vertices are defined in Figure~\ref{2-examples}), 
this represents the following equation (Relation~\ref{relation3}): 
\begin{align*}
M_{4}\cdot M_{1}=\Delta_{\{2,3,6\}}+\Delta_{\{3,5,6\}},
\end{align*} 
where $I=\{1,4\}\subset V_{4}$ (for $n=2$).}
\label{fig_rel3}
\end{figure}


\begin{figure}[H]
\begin{tikzpicture}
\begin{scope}[xscale=0.25, yscale=0.25]
\coordinate (1) at (-1.5,4);
\coordinate (2) at (-4,1.5);
\coordinate (3) at (-4,-1.5);
\coordinate (4) at (-1.5,-4);
\coordinate (5) at (1.5,4);
\coordinate (6) at (4,1.5);
\coordinate (7) at (4,-1.5);
\coordinate (8) at (1.5,-4);

\node[above] at (0,4) {$M_{1}$}; 

\fill(1) circle (5pt);
\fill(2) circle (5pt);
\fill(3) circle (5pt);
\fill(4) circle (5pt);
\fill(5) circle (5pt);
\fill(6) circle (5pt);
\fill(7) circle (5pt);
\fill(8) circle (5pt);

\draw[dashed] (1)--(2);
\draw[dashed] (1)--(3);
\draw[dashed] (1)--(4);
\draw[dashed] (1)--(5);
\draw[dashed] (1)--(6);
\draw[dashed] (1)--(7);
\draw[very thick] (2)--(3);
\draw[very thick] (2)--(4);
\draw[very thick] (2)--(5);
\draw[very thick] (2)--(6);
\draw[very thick] (2)--(8);
\draw[very thick] (3)--(4);
\draw[very thick] (3)--(5);
\draw[very thick] (3)--(7);
\draw[very thick] (3)--(8);
\draw[very thick] (4)--(6);
\draw[very thick] (4)--(7);
\draw[very thick] (4)--(8);
\draw[very thick] (5)--(6);
\draw[very thick] (5)--(7);
\draw[very thick] (5)--(8);
\draw[very thick] (6)--(7);
\draw[very thick] (6)--(8);
\draw[very thick] (7)--(8);

\node at (6,0) {$\times$};

\coordinate (a1) at (10.5,4);
\coordinate (a2) at (8,1.5);
\coordinate (a3) at (8,-1.5);
\coordinate (a4) at (10.5,-4);
\coordinate (a5) at (13.5,4);
\coordinate (a6) at (16,1.5);
\coordinate (a7) at (16,-1.5);
\coordinate (a8) at (13.5,-4);

\node[above] at (12,4) {$M_{2}$};

\fill(a1) circle (5pt);
\fill(a2) circle (5pt);
\fill(a3) circle (5pt);
\fill(a4) circle (5pt);
\fill(a5) circle (5pt);
\fill(a6) circle (5pt);
\fill(a7) circle (5pt);
\fill(a8) circle (5pt);

\draw[dashed] (a1)--(a2);
\draw[very thick] (a1)--(a3);
\draw[very thick] (a1)--(a4);
\draw[very thick] (a1)--(a5);
\draw[very thick] (a1)--(a6);
\draw[very thick] (a1)--(a7);
\draw[dashed] (a2)--(a3);
\draw[dashed] (a2)--(a4);
\draw[dashed] (a2)--(a5);
\draw[dashed] (a2)--(a6);
\draw[dashed] (a2)--(a8);
\draw[very thick] (a3)--(a4);
\draw[very thick] (a3)--(a5);
\draw[very thick] (a3)--(a7);
\draw[very thick] (a3)--(a8);
\draw[very thick] (a4)--(a6);
\draw[very thick] (a4)--(a7);
\draw[very thick] (a4)--(a8);
\draw[very thick] (a5)--(a6);
\draw[very thick] (a5)--(a7);
\draw[very thick] (a5)--(a8);
\draw[very thick] (a6)--(a7);
\draw[very thick] (a6)--(a8);
\draw[very thick] (a7)--(a8);

\node at (18,0) {$\times$};

\coordinate (b1) at (22.5,4);
\coordinate (b2) at (20,1.5);
\coordinate (b3) at (20,-1.5);
\coordinate (b4) at (22.5,-4);
\coordinate (b5) at (25.5,4);
\coordinate (b6) at (28,1.5);
\coordinate (b7) at (28,-1.5);
\coordinate (b8) at (25.5,-4);

\node[above] at (24,4) {$M_{3}$};

\fill(b1) circle (5pt);
\fill(b2) circle (5pt);
\fill(b3) circle (5pt);
\fill(b4) circle (5pt);
\fill(b5) circle (5pt);
\fill(b6) circle (5pt);
\fill(b7) circle (5pt);
\fill(b8) circle (5pt);

\draw[very thick] (b1)--(b2);
\draw[dashed] (b1)--(b3);
\draw[very thick] (b1)--(b4);
\draw[very thick] (b1)--(b5);
\draw[very thick] (b1)--(b6);
\draw[very thick] (b1)--(b7);
\draw[dashed] (b2)--(b3);
\draw[very thick] (b2)--(b4);
\draw[very thick] (b2)--(b5);
\draw[very thick] (b2)--(b6);
\draw[very thick] (b2)--(b8);
\draw[dashed] (b3)--(b4);
\draw[dashed] (b3)--(b5);
\draw[dashed] (b3)--(b7);
\draw[dashed] (b3)--(b8);
\draw[very thick] (b4)--(b6);
\draw[very thick] (b4)--(b7);
\draw[very thick] (b4)--(b8);
\draw[very thick] (b5)--(b6);
\draw[very thick] (b5)--(b7);
\draw[very thick] (b5)--(b8);
\draw[very thick] (b6)--(b7);
\draw[very thick] (b6)--(b8);
\draw[very thick] (b7)--(b8);

\node at (30,0) {$=$};

\coordinate (c1) at (34.5,4);
\coordinate (c2) at (32,1.5);
\coordinate (c3) at (32,-1.5);
\coordinate (c4) at (34.5,-4);
\coordinate (c5) at (37.5,4);
\coordinate (c6) at (40,1.5);
\coordinate (c7) at (40,-1.5);
\coordinate (c8) at (37.5,-4);

\node[above] at (36,4) {$\Delta_{\{5,6,7,8\}}$};

\fill(c1) circle (5pt);
\fill(c2) circle (5pt);
\fill(c3) circle (5pt);
\fill(c4) circle (5pt);
\fill(c5) circle (5pt);
\fill(c6) circle (5pt);
\fill(c7) circle (5pt);
\fill(c8) circle (5pt);

\draw[dashed] (c1)--(c2);
\draw[dashed] (c1)--(c3);
\draw[dashed] (c1)--(c4);
\draw[dashed] (c1)--(c5);
\draw[dashed] (c1)--(c6);
\draw[dashed] (c1)--(c7);
\draw[dashed] (c2)--(c3);
\draw[dashed] (c2)--(c4);
\draw[dashed] (c2)--(c5);
\draw[dashed] (c2)--(c6);
\draw[dashed] (c2)--(c8);
\draw[dashed] (c3)--(c4);
\draw[dashed] (c3)--(c5);
\draw[dashed] (c3)--(c7);
\draw[dashed] (c3)--(c8);
\draw[dashed] (c4)--(c6);
\draw[dashed] (c4)--(c7);
\draw[dashed] (c4)--(c8);
\draw[very thick] (c5)--(c6);
\draw[very thick] (c5)--(c7);
\draw[very thick] (c5)--(c8);
\draw[very thick] (c6)--(c7);
\draw[very thick] (c6)--(c8);
\draw[very thick] (c7)--(c8);

\node at (42,0) {$+$};

\coordinate (d1) at (46.5,4);
\coordinate (d2) at (44,1.5);
\coordinate (d3) at (44,-1.5);
\coordinate (d4) at (46.5,-4);
\coordinate (d5) at (49.5,4);
\coordinate (d6) at (52,1.5);
\coordinate (d7) at (52,-1.5);
\coordinate (d8) at (49.5,-4);

\node[above] at (48,4) {$\Delta_{\{4,6,7,8\}}$};

\fill(d1) circle (5pt);
\fill(d2) circle (5pt);
\fill(d3) circle (5pt);
\fill(d4) circle (5pt);
\fill(d5) circle (5pt);
\fill(d6) circle (5pt);
\fill(d7) circle (5pt);
\fill(d8) circle (5pt);

\draw[dashed] (d1)--(d2);
\draw[dashed] (d1)--(d3);
\draw[dashed] (d1)--(d4);
\draw[dashed] (d1)--(d5);
\draw[dashed] (d1)--(d6);
\draw[dashed] (d1)--(d7);
\draw[dashed] (d2)--(d3);
\draw[dashed] (d2)--(d4);
\draw[dashed] (d2)--(d5);
\draw[dashed] (d2)--(d6);
\draw[dashed] (d2)--(d8);
\draw[dashed] (d3)--(d4);
\draw[dashed] (d3)--(d5);
\draw[dashed] (d3)--(d7);
\draw[dashed] (d3)--(d8);
\draw[very thick] (d4)--(d6);
\draw[very thick] (d4)--(d7);
\draw[very thick] (d4)--(d8);
\draw[dashed] (d5)--(d6);
\draw[dashed] (d5)--(d7);
\draw[dashed] (d5)--(d8);
\draw[very thick] (d6)--(d7);
\draw[very thick] (d6)--(d8);
\draw[very thick] (d7)--(d8);
\end{scope}
\end{tikzpicture}
\caption{
For the GKM graph $\mathcal{GQ}_{6}$ (where the vertices are defined in Figure~\ref{2-examples}), 
this represents the following equation (Relation~\ref{relation3}): 
\begin{align*}
M_{1}\cdot M_{2}\cdot M_{3}=\Delta_{\{5,6,7,8\}}+\Delta_{\{4,6,7,8\}},
\end{align*} 
where $I=\{1,2,3\}\subset V_{6}$ (for $n=3$).}
\label{fig_rel3-ver2}
\end{figure}

\begin{relation}[$\Delta_{K}\cdot M_{i}=\Delta_{K\setminus \{i\}}$]
\label{relation4}
For two generators $\Delta_{K}$ and $M_{i}$, we have the following relation (see Figure~\ref{fig_rel4}).
\begin{lemma}[Relation 4]
\label{lem-for-rel4}
Fix $i\in V_{2n}$.
If a subset $K\subset V_{2n}$ satisfies $\{i\}\subsetneq K$ and the property $(\ast)$,  
then there is the following relation:
\begin{align}
\label{4th-rel}
\Delta_{K}\cdot M_{i}=\Delta_{K\setminus \{i\}}.
\end{align}
\end{lemma}
\begin{proof}
The multiplication of $\Delta_{K}$ and $M_{i}$ is not zero only on $K\cap (V_{2n}\setminus \{i\})=K\setminus \{i\}$. 
Therefore, we have the following equations.
\begin{align*}
\Delta_{K}\cdot M_{i}(v)
=
\left\{
\begin{array}{ll}
{\displaystyle \prod_{j\not\in (K\setminus \{i\})\cup \{\overline{v}\} }\alpha(vj)} & \text{if $v\in K\setminus \{i\}$} \\
0 & \text{if $v\not\in K\setminus \{i\}$}
\end{array}
\right.
\end{align*}
This shows the equation $\Delta_{K}\cdot M_{i}=\Delta_{K\setminus \{i\}}$.
\end{proof}
\end{relation}

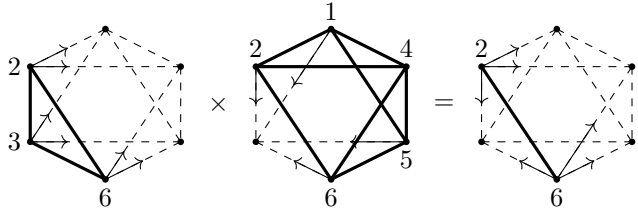
\begin{figure}[H]
\begin{tikzpicture}
\begin{scope}[xscale=0.25, yscale=0.25]
\fill(0,4) circle (5pt);
\fill(-4,2) circle (5pt);
\fill(-4,-2) circle (5pt);
\fill(0,-4) circle (5pt);
\fill(4,-2) circle (5pt);
\fill(4,2) circle (5pt);

\draw[dashed] (0,4)--(-4,2);
\draw[dashed] (0,4)--(-4,-2);
\draw[dashed] (0,4)--(4,-2);
\draw[dashed] (0,4)--(4,2);
\draw[very thick] (-4,2)--(-4,-2);
\draw[very thick] (-4,2)--(0,-4);
\draw[dashed] (-4,2)--(4,2);
\draw[dashed] (-4,-2)--(4,-2);
\draw[very thick] (-4,-2)--(0,-4);
\draw[dashed] (4,2)--(4,-2);
\draw[dashed] (4,2)--(0,-4);
\draw[dashed] (4,-2)--(0,-4);

\draw[->] (-4,2)--(-2,2);
\draw[->] (-4,2)--(-2,3);
\draw[->] (-4,-2)--(-2,-2);
\draw[->] (-4,-2)--(-3,-0.5);
\draw[->] (0,-4)--(1,-2.5);
\draw[->] (0,-4)--(2,-3);

\node[left] at (-4,2) {$2$};
\node[left] at (-4,-2) {$3$};
\node[below] at (0,-4) {$6$};

\node at (6,0) {$\times$};

\fill(12,4) circle (5pt);
\fill(8,2) circle (5pt);
\fill(8,-2) circle (5pt);
\fill(12,-4) circle (5pt);
\fill(16,-2) circle (5pt);
\fill(16,2) circle (5pt);

\draw[very thick] (12,4)--(8,2);
\draw[dashed] (12,4)--(8,-2);
\draw[very thick] (12,4)--(16,-2);
\draw[very thick] (12,4)--(16,2);
\draw[dashed] (8,2)--(8,-2);
\draw[very thick] (8,2)--(12,-4);
\draw[very thick] (8,2)--(16,2);
\draw[dashed] (8,-2)--(16,-2);
\draw[dashed] (8,-2)--(12,-4);
\draw[very thick] (16,2)--(16,-2);
\draw[very thick] (16,2)--(12,-4);
\draw[very thick] (16,-2)--(12,-4);

\draw[->] (12,4)--(10,1);
\draw[->] (8,2)--(8,0);
\draw[->] (12,-4)--(10,-3);
\draw[->] (16,-2)--(13,-2);

\node[above] at (12,4) {$1$};
\node[above] at (8,2) {$2$};
\node[below] at (12,-4) {$6$};
\node[above] at (16,2) {$4$};
\node[below] at (16,-2) {$5$};

\node at (18,0) {$=$};

\fill(24,4) circle (5pt);
\fill(20,2) circle (5pt);
\fill(20,-2) circle (5pt);
\fill(24,-4) circle (5pt);
\fill(28,-2) circle (5pt);
\fill(28,2) circle (5pt);

\draw[dashed] (24,4)--(20,2);
\draw[dashed] (24,4)--(20,-2);
\draw[dashed] (24,4)--(28,-2);
\draw[dashed] (24,4)--(28,2);
\draw[dashed] (20,2)--(20,-2);
\draw[very thick] (20,2)--(24,-4);
\draw[dashed] (20,2)--(28,2);
\draw[dashed] (20,-2)--(28,-2);
\draw[dashed] (20,-2)--(24,-4);
\draw[dashed] (28,2)--(28,-2);
\draw[dashed] (28,2)--(24,-4);
\draw[dashed] (28,-2)--(24,-4);

\node[below] at (24,-4) {$6$};
\node[above] at (20,2) {$2$}; 

\draw[->] (20,2)--(22,2);
\draw[->] (20,2)--(20,0);
\draw[->] (20,2)--(22,3);

\draw[->] (24,-4)--(22,-3);
\draw[->] (24,-4)--(26,-3);
\draw[->] (24,-4)--(26,-1);

\end{scope}
\end{tikzpicture}
\caption{$\Delta_{\{2,3,6\}}\cdot M_{3}=\Delta_{\{2,6\}}$ (Relation~\ref{relation4}), where $i=3$ and 
$K=\{2,3,6\}$.} 
\label{fig_rel4}
\end{figure}

\section{Main theorem and its proof}
\label{sect:5}
In this section, we prove the main theorem (Theorem~\ref{main-theorem2}). 
To state the main theorem precisely, we first prepare some notations.
We denote the set of elements defined in Section~\ref{sect:3} as follows:
\begin{description}
\item[Generator 1] $\mathcal{M}:=\{M_{v}\ |\ v\in V_{2n}\}$;
\item[Generator 2] $\mathcal{D}:=\{\Delta_{K}\ |\ K\subset V_{2n}\ \text{with the property $(\ast)$}\}$.
\end{description}
Let $\mathbb{Z}[\mathcal{M}, \mathcal{D}]$ be the polynomial ring which generated by all elements in $\mathcal{M}$ and $\mathcal{D}$. 
We define the degree of elements by  
\begin{itemize}
\item $\deg M_{v}=2$ for every $M_{v}\in \mathcal{M}$; 
\item $\deg \Delta_{K}=2(2n-(|K|-1))=4n-2|K|+2$ for every $\Delta_{K}\in \mathcal{D}$.
\end{itemize}
Let $\mathcal{I}$ be the ideal in $\mathbb{Z}[\mathcal{M}, \mathcal{D}]$ generated by the four relations defined in Section~\ref{sect:4}.
Namely, 
the ideal $\mathcal{I}$ in $\mathbb{Z}[\mathcal{M}, \mathcal{D}]$ is generated by 
the following four types of elements:
\begin{description}
\item[Relation 1] $\prod_{\cap J=\emptyset}G_{J}$ for $G_{J}\in \mathcal{M}\sqcup \mathcal{D}$;
\item[Relation 2] $(M_{i}+M_{\overline{i}})-(M_{j}+M_{\overline{j}})$ for every distinct $i,j\in V_{2n}$;
\item[Relation 3] $\prod_{i\in I}M_{i}-(\Delta_{(I\cup \{a\})^{c}}+\Delta_{(I\cup \{\overline{a}\})^{c}})$ for every subset $I\subset V_{2n}$ which satisfies the property $(\ast)$ and $|I|=n$, where 
$\{a,\overline{a}\}$ is the unique pair in $V_{2n}\setminus I$;
\item[Relation 4] $\Delta_{K}\cdot M_{i}-\Delta_{K\setminus \{i\}}$ for $\{i\}\subsetneq K$.  
\end{description}
We use the following notation: 
\begin{align}
\label{Face-ring}
\mathbb{Z}[\mathcal{GQ}_{2n}]:=\mathbb{Z}[\mathcal{M},\mathcal{D}]/\mathcal{I}.
\end{align}
Because of Section~\ref{sect:3} and Section~\ref{sect:4}, there exists the well-defined homomorphism 
\begin{align*}
\psi:\mathbb{Z}[\mathcal{GQ}_{2n}]\to H^{*}(\mathcal{GQ}_{2n})
\end{align*}
by the induced homomorphism from 
\begin{align*}
\widetilde{\psi}:\mathbb{Z}[\mathcal{M},\mathcal{D}]\to H^{*}(\mathcal{GQ}_{2n}).
\end{align*}
Namely, 
$\psi$ is induced from 
the following commutative diagram:
\begin{align}
\label{diagram}
\xymatrix{
\mathbb{Z}[\mathcal{M},\mathcal{D}] \ar[d] \ar[rd]^{\widetilde{\psi}} & \\
\mathbb{Z}[\mathcal{GQ}_{2n}]\ar[r]^{\psi} & H^{*}(\mathcal{GQ}_{2n})}
\end{align}
where the vertical map is the natural projection.

The following theorem is the main theorem of this paper. 
\begin{theorem}
\label{main-theorem2}
The homomorphism $\psi$ is the isomorphism, i.e., 
\begin{align*}
\mathbb{Z}[\mathcal{GQ}_{2n}]\simeq H^{*}(\mathcal{GQ}_{2n}).
\end{align*} 
\end{theorem}
Together with Lemma~\ref{key-lemma}, we obtain Theorem~\ref{main}.

Note that in the proofs below,
by definitions of generators in Section~\ref{sect:3}, 
we may write $\widetilde{\psi}(M_{i})=M_{i}, \widetilde{\psi}(\Delta_{K})=\Delta_{K}\in H^{*}(\mathcal{GQ}_{2n})$.

\subsection{Surjectivity of $\psi:\mathbb{Z}[\mathcal{GQ}_{2n}]\to H^{*}(\mathcal{GQ}_{2n})$}
\label{sect:5.1}
We first prove the surjectivity of $\psi$.
To prove it, we use the inductive argument for vertices which is often used in GKM theory (see e.g. \cite[Lemma 4.4]{MMP} or \cite[Lemma 5.6]{KU}).
\begin{lemma}
\label{surjectivity}
The homomorphism 
$\psi:\mathbb{Z}[\mathcal{GQ}_{2n}]\to H^{*}(\mathcal{GQ}_{2n})$
is surjective.
\end{lemma}
\begin{proof}
By the commutative diagram \eqref{diagram}, 
it is enough to prove that $\widetilde{\psi}$ is surjective.
Take an element $f\in H^{*}(\mathcal{GQ}_{2n})$. 
By definition, for the vertex $1\in V_{2n}$, the polynomial $f(1)\in H^{*}(BT^{n+1})$ can be written by
\begin{align*}
f(1)=\sum_{\textbf{a}}k_{\textbf{a}}x^{\textbf{a}}=g_{1}
\end{align*}
where $k_{\textbf{a}}\in \mathbb{Z}$ and $x_{\textbf{a}}:=x_{1}^{a_{1}}\cdots x_{n+1}^{a_{n+1}}$ for $\textbf{a}=(a_{1},\cdots ,a_{n+1})\in (\mathbb{N}\cup \{0\})^{n+1}$.
By definition of $M_{2},\ldots, M_{n+2}$ in Definition~\ref{deg-2-gen},
we have $M_{2}(1)=x_{1},\ldots, M_{n+2}(1)=x_{n+1}$; therefore, 
\begin{align*}
x^{\textbf{a}}=x_{1}^{a_{1}}\cdots x_{n+1}^{a_{n+1}}=M_{2}^{a_{1}}\cdots M_{n+1}^{a_{n}}M_{n+2}^{a_{n+1}}(1).
\end{align*}
This means that we may take an element from $\mathbb{Z}[M_{2},\ldots, M_{n+2}]\subset \mathbb{Z}[\mathcal{M},\mathcal{D}]$ whose image of $\widetilde{\psi}$ coincides with $f(1)$ on the vertex $1\in V$.

We next put 
\begin{align*}
f_{2}=f-g_{1}. 
\end{align*} 
Then, $f_{2}(1)=0$.
So, 
by the congruence relations on the edge $21\in E_{2n}$, 
we have 
\begin{align*}
f_{2}(2)-f_{2}(1)\equiv 0 \mod \alpha(21)=M_{1}(2).
\end{align*}
Therefore, we have that 
$f_{2}(2)=g_{2}M_{1}(2)$ for some $g_{2}\in H^{*}(BT^{n+1})$.
By Proposition~\ref{polynomial generators}, 
we have that 
\begin{align*}
x_{1}=M_{2}-M_{1},\ \ldots,\ x_{n+1}=M_{n+2}-M_{1}.
\end{align*}
This implies that 
$g_{2}\in \mathbb{Z}[M_{2}-M_{1},M_{3}-M_{1},\ldots, M_{n+2}-M_{1}]\subset \mathbb{Z}[\mathcal{M},\mathcal{D}]$.
Note that it may also be regarded as $g_{2}\in \mathbb{Z}[M_{1},M_{2},M_{3},\ldots, M_{n+2}]$.
This shows that $g_{2}M_{1}$ is in the image of $\widetilde{\psi}$.
Put 
\begin{align*}
f_{3}=f_{2}-g_{2}M_{1} (=f-g_{1}-g_{2}M_{1}).
\end{align*}
Then, by $f_{2}(1)=M_{1}(1)=0$ and $f_{2}(2)=g_{2}M_{1}(2)$, we have 
\begin{align*}
f_{3}(1)=f_{3}(2)=0.
\end{align*}
By the similar argument, we may write $f_{3}(3)=g_{3}M_{1}M_{2}(3)$ 
for some $g_{3}\in \mathbb{Z}[M_{1},M_{2},M_{3},\ldots, M_{n+2}]$.
Similarly, we can also check that 
$f_{4}:=f_{3}-g_{3}M_{1}M_{2}$ satisfies $f_{4}(1)=f_{4}(2)=f_{4}(3)=0$.
Iterating similar arguments $n+2$ times (note that $\overline{n+1}=n+2$), 
we obtain an element 
\begin{align*}
f_{n+2}:=f_{n+1}-g_{n+1}M_{1}\cdots M_{n}\in H^{*}(\mathcal{GQ}_{2n})
\end{align*}
such that $g_{n+1}\in \mathbb{Z}[M_{1},M_{2},\ldots, M_{n+2}]$ and 
 $f_{n+1}\in H^{*}(\mathcal{GQ}_{2n})$ satisfies that 
$f_{n+1}(1)=\cdots =f_{n+1}(n)=0$ and $f_{n+1}(n+1)=g_{n+1}M_{1}\cdots M_{n}(n+1)$.
Consequently, 
we have that 
\begin{align}
\label{1st_part_surj}
f_{n+2}&=f_{n+1}-g_{n+1}M_{1}\cdots M_{n} \\
&= f_{n}-(g_{n}M_{1}\cdots M_{n-1}+g_{n+1}M_{1}\cdots M_{n}) \nonumber \\
&\quad \vdots \nonumber \\
&=f-(g_{1}+g_{2}M_{1}+\cdots +g_{n}M_{1}\cdots M_{n-1} +g_{n+1}M_{1}\cdots M_{n}). \nonumber
\end{align}

Note that $f_{n+2}$ satisfies that $f_{n+2}(1)=\cdots=f_{n+2}(n+1)=0$.
Therefore, by the definition of $\Delta_{\{n+2,\ldots, 2n+2\}}$ and the congruence relation (see \eqref{graph-equivariant-cohomology}), there exists an element $g_{n+2}\in \mathbb{Z}[M_{1},M_{2},\ldots, M_{n+2}]$ such that  
\begin{align*}
f_{n+2}(n+2)=g_{n+2}\Delta_{\{n+2,\ldots, 2n+2\}}(n+2).
\end{align*}
Since $\Delta_{\{n+2,\ldots, 2n+2\}}(1)=\cdots =\Delta_{\{n+2,\ldots, 2n+2\}}(n+1)=0$, if we put $f_{n+3}:=f_{n+2}-g_{n+2}\Delta_{\{n+2,\ldots, 2n+2\}}$, then 
\begin{align*}
f_{n+3}(1)=\cdots =f_{n+3}(n+2)=0.
\end{align*}
Similarly, for $k\ge 2$, 
there exists $g_{n+k}\in \mathbb{Z}[M_{1},M_{2},\ldots, M_{n+2}]$ such that 
\begin{align}
\label{2nd_part_surj}
f_{n+k+1}:=f_{n+k}-g_{n+k}\Delta_{\{n+k,\ldots, 2n+2\}} 
\end{align}
and 
$f_{n+k+1}(1)=\cdots =f_{n+k+1}(n+k)=0$. 
Then, in the case when $k=n+2$, 
there exists $g_{2n+2}\in \mathbb{Z}[M_{1},M_{2},\ldots, M_{n+2}]$ such that 
\begin{align*}
f_{2n+3}:=f_{2n+2}-g_{2n+2}\Delta_{\{2n+2\}} 
\end{align*}
and $f_{2n+3}(v)=0$ for all $v\in V_{2n}$.
Therefore, $f_{2n+2}=g_{2n+2}\Delta_{\{2n+2\}}$. 
Substituting this equation to \eqref{2nd_part_surj} for $k=n+1$, we get $f_{2n+1}$; and iterating this argument from $k=n$ to $k=2$, 
we have 
\begin{align*}
f_{2n+1}&=g_{2n+2}\Delta_{\{2n+2\}}+g_{2n+1}\Delta_{\{2n+1, 2n+2\}}; \\ 
f_{2n}&=g_{2n+2}\Delta_{\{2n+2\}}+g_{2n+1}\Delta_{\{2n+1, 2n+2\}}+g_{2n}\Delta_{\{2n,2n+1, 2n+2\}}; \\
&\quad \vdots \\
f_{n+2}&=g_{2n+2}\Delta_{\{2n+2\}}+g_{2n+1}\Delta_{\{2n+1, 2n+2\}}+\cdots + g_{n+2}\Delta_{\{n+2,\ldots, 2n+2\}}.
\end{align*} 
Together with \eqref{1st_part_surj}, every element  
$f\in H^{*}(\mathcal{GQ}_{2n})$ can be written by the elements in $\mathbb{Z}[\mathcal{M},\mathcal{D}]$ as follows:
\begin{align}
\label{linear-relation}
f=g_{1}+g_{2}M_{1}+\cdots &+g_{n+1}M_{1}\cdots M_{n} \\
&+g_{n+2}\Delta_{\{n+2,\ldots, 2n+2\}}+g_{n+3}\Delta_{\{n+3,\ldots, 2n+2\}}+\cdots 
+g_{2n+2}\Delta_{\{2n+2\}} \nonumber
\end{align}
for some 
$g_{1}\in \mathbb{Z}[M_{2},\ldots, M_{n+2}]$ and 
$g_{2},\ldots, g_{2n+2}\in \mathbb{Z}[x_{1},\ldots, x_{n+1}]\simeq \mathbb{Z}[M_{2}-M_{1},\ldots, M_{n+2}-M_{1}]\subset \mathbb{Z}[M_{1},\ldots, M_{n+2}]\subset \mathbb{Z}[\mathcal{M},\mathcal{D}]$.
This shows the subjectivity of $\widetilde{\psi}:\mathbb{Z}[\mathcal{M},\mathcal{D}]\to H^{*}(\mathcal{GQ}_{2n})$.
\end{proof}

\subsection{Injectivity of $\psi:\mathbb{Z}[\mathcal{GQ}_{2n}]\to H^{*}(\mathcal{GQ}_{2n})$}
\label{sect:5.2}

We next prove the injectivity of $\psi$.
\begin{lemma}
\label{injectivity}
The homomorphism 
$\psi:\mathbb{Z}[\mathcal{GQ}_{2n}]\to H^{*}(\mathcal{GQ}_{2n})$
is injective.
\end{lemma}

To show this lemma, we will use the combinatorial counterpart of the localization theorem which will be stated in Corollary~\ref{localization_theorem}.
To state that, we prepare the following notation.
For $v\in V_{2n}=[2n+2]$, the subset $I_{v}\subset [n+2]\subset V_{2n}$ is defined by 
\begin{itemize}
\item $I_{v}=[n+2]\setminus \{v\}$ for $1\le v\le n+1$;
\item $I_{v}=[n+2]\setminus \{\overline{v}\}$ for $n+2\le v\le 2n+2$.
\end{itemize} 
Note that $I_{v}=I_{2n+3-v}$ for $1\le v\le n+1$, for example, $I_{1}=\{2,\ldots, n+2\}=I_{2n+2}$.

\begin{lemma}
\label{localization_1st}
The following isomorphism holds for every $v\in V_{2n}$:
\begin{align*}
\mathbb{Z}[\mathcal{GQ}_{2n}]/\langle G_{J}\ |\ v\not\in J \rangle\simeq\mathbb{Z}[M_{i}\ |\ i\in I_{v}]
\simeq H^{*}(BT^{n+1}),
\end{align*}
where $\langle G_{J}\ |\ v\not\in J \rangle$ is the ideal in $\mathbb{Z}[\mathcal{GQ}_{2n}]$ generated by $G_{J}$ with $v\not\in J$ (see Relation~\ref{relation1} in Section~\ref{sect:4}).
\end{lemma}
\begin{proof}
We will prove the statement only for the vertex $v=1\in V_{2n}$ because the proofs for the other vertices in $V_{2n}=[2n+2]=\{1,\ldots, 2n+2\}$ are similar. 

Suppose that $v=1\in V_{2n}$. 
We shall prove that 
\begin{align*}
\mathbb{Z}[\mathcal{GQ}_{2n}]/\langle G_{J}\ |\ 1\not\in J \rangle\simeq 
\mathbb{Z}[M_{2},\ldots, M_{n+1},M_{n+2}]\simeq H^{*}(BT^{n+1}).
\end{align*}

We first claim that every element in $\mathcal{D}$ can be written by the elements in $\mathcal{M}$ in $\mathbb{Z}[\mathcal{GQ}_{2n}]/\langle G_{J}\ |\ 1\not\in J \rangle$.
Assume that $K\subset V_{2n}$ satisfies $\{i, \overline{i}\}\not\subset K$ for every $i=1,\ldots, n+1$.
If $1\not\in K$, then $\Delta_{K}=0$ in $\mathbb{Z}[\mathcal{GQ}_{2n}]/\langle G_{J}\ |\ 1\not\in J \rangle$.
If $1\in K$ and $|K|< n+1$, then 
by Relation 4, we have that 
\begin{align*}
\Delta_{K}=\Delta_{K\cup \{j\}}\cdot M_{j}\ \text{for $j, \overline{j}\not\in K$}.
\end{align*}
This implies that in $\mathbb{Z}[\mathcal{GQ}_{2n}]/\langle G_{J}\ |\ 1\not\in J \rangle$ the generators in $\mathcal{D}$ can be written by 
$\Delta_{K}$'s 
such that $1\in K$ and $|K|=n+1$.
We next assume that $I\subset V_{2n}$ satisfies $|I|=n$ 
such that $1\not\in I$ and 
there is the unique pair $\{a,\overline{a}\}\subset V_{2n}\setminus I$.
Put $I=\{j_{1},\ldots, j_{n}\}$ and $I^{c}=V_{2n}\setminus I=\{1, i_{1},\ldots, i_{n-1}, a, \overline{a}\}$.
Then, by Relation 3, we have 
\begin{align*}
\Delta_{\{1,i_{1},\ldots, i_{n-1}, \overline{a}\}}=M_{j_{1}}\cdots M_{j_{n}}-\Delta_{\{1,i_{1},\ldots, i_{n-1}, a\}}\in \mathbb{Z}[\mathcal{GQ}_{2n}]/\langle G_{J}\ |\ 1\not\in J \rangle.
\end{align*}
This shows that for the generator $\Delta_{K}\in \mathcal{D}$ such that $1\in K$ and $|K|=n+1$, if there is the vertex $\overline{a}\in K$ for $a=2,\ldots n+1$, then we may replace $\Delta_{K}$ into $\Delta_{(K\setminus \{\overline{a}\})\cup \{a\}}$ by using elements in $\mathcal{M}$. 
Therefore, we may reduce the generators in $\mathcal{D}$ into only one generator 
$\Delta_{\{1,\ldots, n+1\}}$. 
Moreover, 
since $1\not\in \{2,\ldots, n+1,2n+2\}$, we have 
\begin{align*}
M_{n+2}\cdots M_{2n+1}=\Delta_{\{1,2,\ldots, n+1\}}+\Delta_{\{2,\ldots, n+1,2n+2\}}
=\Delta_{\{1,2,\ldots, n+1\}}\in \mathbb{Z}[\mathcal{GQ}_{2n}]/\langle G_{J}\ |\ 1\not\in J \rangle,
\end{align*}
by Relation 3. 
This shows that every element in $\mathcal{D}$ can be written by the elements in $\mathcal{M}$.

Next, 
by the definition of $M_{1}$, we have $M_{1}=G_{V_{2n}\setminus\{1\}}=0$ 
in $\mathbb{Z}[\mathcal{GQ}_{2n}]/\langle G_{J}\ |\ 1\not\in J \rangle$.
Therefore, together with Relation 2, we have that 
\begin{align*}
M_{\overline{1}}=M_{2}+M_{\overline{2}}=\cdots =M_{n+1}+M_{\overline{n+1}}\in \mathbb{Z}[\mathcal{GQ}_{2n}]/\langle G_{J}\ |\ 1\not\in J \rangle.
\end{align*}
Therefore, $M_{\overline{1}}=M_{n+1}+M_{n+2}$ and $M_{\overline{k}}=M_{n+1}+M_{n+2}-M_{k}$ for $k=2,\ldots, n$.
This implies that the generators in $\mathcal{M}$ can be reduced into 
\begin{align*}
M_{2}, \ldots, M_{n+1}, M_{n+2}. 
\end{align*}
This shows that there is the surjective homomorphism 
\begin{align*}
p:\mathbb{Z}[M_{2},\ldots, M_{n+2}]\to \mathbb{Z}[\mathcal{GQ}_{2n}]/\langle G_{I}\ |\ 1\not\in I \rangle
\end{align*}
defined by $p(M_{i})=M_{i}$ for $i=2,\ldots, n+2$.
We finally consider the following composition homomorphism:
\begin{align*}
\mathbb{Z}[M_{2},\ldots, M_{n+2}]\stackrel{p}{\to} \mathbb{Z}[\mathcal{GQ}_{2n}]/\langle G_{I}\ |\ 1\not\in I \rangle 
\stackrel{\iota_{1}}{\to} H^{*}(BT^{n+1}),
\end{align*}
where 
$\iota_{1}$ is the induced homomorphism from 
$\mathbb{Z}[\mathcal{GQ}_{2n}]\stackrel{\psi}{\to} H^{*}(\mathcal{GQ}_{2n})\to H^{*}(BT^{n+1})$ such that $f\mapsto \psi(f)(1)$ for $f\in \mathbb{Z}[\mathcal{GQ}_{2n}]$. 
By the definition of $M_{i}$, we have $\iota_{1}\circ p(M_{i})=x_{i-1}$ for $i=2,\ldots, n+2$.
Therefore, the composition map $\iota_{1}\circ p$ is an isomorphism.
This shows that $p$ is injective.
Consequently, $p$ is an isomorphism.
This establishes that  
$\mathbb{Z}[\mathcal{GQ}_{2n}]/\langle G_{I}\ |\ 1\not\in I \rangle\simeq  \mathbb{Z}[M_{2},\ldots, M_{n+2}]\simeq  H^{*}(BT^{n+1})$.
\end{proof}


Therefore, by the definition of the graph equivariant cohomology and Lemma~\ref{localization_1st},  we have the following corollary.
\begin{corollary}
\label{localization_theorem}
There is an injective homomorphism:
\begin{align*}
H^{*}(\mathcal{GQ}_{2n})\hookrightarrow \bigoplus_{v\in V_{2n}}H^{*}(BT^{n+1})\simeq \bigoplus_{v\in V_{2n}}\mathbb{Z}[\mathcal{GQ}_{2n}]/\langle G_{J}\ |\ v\not\in J \rangle\simeq \bigoplus_{v\in V_{2n}}\mathbb{Z}[M_{i}\ |\ i\in I_{v}].
\end{align*}
\end{corollary}

Notice that Corollary~\ref{localization_theorem} may be regarded as the counterpart of the localization theorem for the usual equivariant cohomology.


Now we may prove Lemma~\ref{injectivity}.

\begin{proof}[Proof of Lemma~\ref{injectivity}]
It is enough to prove that the following composition map $\varphi$ is injective:
\begin{align*}
\varphi:\mathbb{Z}[\mathcal{GQ}_{2n}]\stackrel{\psi}{\longrightarrow} H^{*}(\mathcal{GQ}_{2n})\hookrightarrow \bigoplus_{v\in V_{2n}}H^{*}(BT^{n+1}) 
\simeq \bigoplus_{v\in V_{2n}}\mathbb{Z}[M_{i}\ |\ i\in I_{v}].
\end{align*}
Assume that $\varphi(f)=0$ for an element $f\in \mathbb{Z}[\mathcal{GQ}_{2n}]$. 
We will prove that $f=0$.
In the proof, we use the following restriction map for $w\in V_{2n}$: 
\begin{align*}
\rho_{w}:
\bigoplus_{v\in V_{2n}}\mathbb{Z}[M_{i}\ |\ i\in I_{v}]
\to \mathbb{Z}[M_{i}\ |\ i\in I_{w}]
\end{align*}
and the image of $f\in \mathbb{Z}[\mathcal{GQ}_{2n}]$ by the composition map $\rho_{w}\circ \varphi$ by $f(w)(:=\rho_{w}\circ \varphi(f))$.
The assumption $\varphi(f)=0$ is equivalent to that $\rho_{v}\circ \varphi(f)=f(v)=0\in \mathbb{Z}[M_{i}\ |\ i\in I_{v}]$ for all $v\in V_{2n}$.

In the proof of Lemma~\ref{surjectivity}, especially in Equation~\eqref{linear-relation}, we also show the following fact: for any element $f\in \mathbb{Z}[\mathcal{GQ}_{2n}]$, there exists $g_{i},g_{n}'\in 
\mathbb{Z}[M_{2}-M_{1},\ldots, M_{n+2}-M_{1}] \subset \mathbb{Z}[M_{1},\ldots, M_{n+2}]$ for $i=1,\ldots, 2n$ and $g_{0}\in \mathbb{Z}[M_{2},\ldots, M_{n+2}]$ such that  
\begin{align}
\label{induction_injective}
f&=g_{0}+g_{1}M_{1}+\cdots +g_{n}M_{1}\cdots M_{n}+g_{n}'\Delta_{\{n+2,\ldots, 2n+2\}}+g_{n+1}\Delta_{\{n+3,\ldots, 2n+2\}}+\cdots 
+g_{2n}\Delta_{\{2n+2\}} \\
&=g_{0}+\sum_{i=1}^{n}{g_{i}}M_{1}\cdots M_{i}+X(\Delta), \nonumber
\end{align}
where $X(\Delta)$ is the $\Delta_{K}$ terms.
Note that $\psi(g_{j}), \psi(g_{n}')\in \mathbb{Z}[x_{1},\ldots, x_{n+1}]$ (see \eqref{algebra-structure}) for all $j=0,\ldots, 2n$.
This implies that if there is a vertex $v\in V_{2n}$ such that $g_{j}(v)=0$ (resp.~$g_{n}'(v)=0$), then $g_{j}=0$ (resp.~$g_{n}'=0$).

We first claim that $f$ can be written by $\Delta_{K}$ terms only.
Since $M_{1}(1)=0$ and $X(\Delta)(1)=0$, 
by \eqref{induction_injective}, we have that 
\begin{align*}
g_{0}(1)=f(1)-\left(\sum_{i=1}^{n}{g_{i}}M_{1}\cdots M_{i}+X(\Delta)\right)(1)=f(1)=0.
\end{align*}
Therefore, we have $g_{0}=0$.
Similarly, by using $g_{0}=0$ and \eqref{induction_injective}, we have that 
\begin{align*}
g_{1}(2)M_{1}(2)=f(2)-\left(\sum_{i=2}^{n}{g_{i}}M_{1}\cdots M_{i}+X(\Delta)\right)(2)=0.
\end{align*}
Now $g_{1}(2), M_{1}(2)\in \mathbb{Z}[M_{i}\ |\ i\in I_{2}=\{1,3,\ldots,n+2\}]$ and $M_{1}(2)\not=0$.
Since the polynomial ring $\mathbb{Z}[M_{i}\ |\ i\in I_{2}]$ is an integral domain, we see that $g_{1}(2)=0$; therefore, $g_{1}=0$.
Iterating the similar arguments for $i=3,\ldots, n-2$, 
we also have that $g_{2}=\cdots =g_{n-1}=0$, i.e., 
\begin{align*}
f&=g_{n}M_{1}\cdots M_{n}+X(\Delta) \\
&=g_{n}(\Delta_{\{n+2,\ldots, 2n+2\}}+\Delta_{\{n+1,n+3,\ldots, 2n+2\}})+X(\Delta) \quad ({\rm by\ Relation~3}). 
\end{align*}
Therefore, if $f(v)=0$ for every $v\in V_{2n}$, then $f$ can be written by the $\Delta_{K}$ terms only; more precisely, 
\begin{align}
\label{middle-eq}
f=g_{n}\Delta_{\{n+1,n+3,\ldots, 2n+2\}}+(g_{n}+g_{n}')\Delta_{\{n+2,\ldots, 2n+2\}}+g_{n+1}\Delta_{\{n+3,\ldots, 2n+2\}}+\cdots 
+g_{2n}\Delta_{\{2n+2\}}.
\end{align}

We next claim that $f=0$ if $f(v)=0$ for every $v\in V_{2n}$.
The equality \eqref{middle-eq} implies that for the vertex $n+1\in V_{2n}$,
\begin{align*}
&g_{n}(n+1)\Delta_{\{n+1,n+3,\ldots, 2n+2\}}(n+1) \\
&=f(n+1)-\left((g_{n}+g_{n}')\Delta_{\{n+2,\ldots, 2n+2\}}+g_{n+1}\Delta_{\{n+3,\ldots, 2n+2\}}+\cdots 
+g_{2n}\Delta_{\{2n+2\}}\right)(n+1)=0.
\end{align*}
Since $\Delta_{\{n+1,n+3,\ldots, 2n+2\}}(n+1)\not=0$, by the similar reason as above, we have $g_{n}=0$.
Iterating the similar arguments for $i=n+2,\ldots, 2n+2$, 
we have that $g_{n}'=g_{n+1}=\cdots =g_{2n}=0$.
This establishes that $f=0$.
Consequently, $\varphi$ is injective.
\end{proof}

\section{Multiplicative formula of $\Delta_{K}, \Delta_{H}$ with $|K|=|H|=n+1$}
\label{sect:6}

In this section, we show some multiplicative formula in $H^{*}(\mathcal{GQ}_{2n})$ which gives a typical difference between $H^{*}(\mathcal{GQ}_{2n})$ and the graph equivariant cohomology ring of a torus graph, i.e., the face ring proved in \cite{MMP}.

Let $K, H\subset V_{2n}$ be the subsets with the property $(\ast)$ and $|K|=|H|=n+1$,
i.e., there are classes $\Delta_{K}, \Delta_{H}\in H^{2n}(\mathcal{GQ}_{2n})$.
Note that if $K\cap H\not=\emptyset$, then we can also define $\Delta_{K\cap H}\in H^{4n-2k}(\mathcal{GQ}_{2n})$ for $k=|K\cap H|-1$.
If $K\cap H=\emptyset$, then we put $\Delta_{\emptyset}=0$.
Recall that the elementary symmetric polynomial with degree $j$ is defined by
\begin{align*}
\mathfrak{S}_{j}(r_{i}\ |\ i=1,\ldots, n):=\sum_{\substack{a_{1}+\cdots+a_{n}=j, \\ 0\le a_{i}\le 1}}r_{1}^{a_{1}}\cdots r_{n}^{a_{n}}.
\end{align*}
Moreover, 
because of Relation 2,  for every $v=1,\ldots, n+1$,
we may put 
\begin{align*}
X:=M_{v}+M_{\overline{v}}\in H^{2}(\mathcal{GQ}_{2n}).
\end{align*}
There is the following multiplicative formula in $H^{*}(\mathcal{GQ}_{2n})\simeq \mathbb{Z}[\mathcal{GQ}_{2n}]$ (see Figure~\ref{fig_rel5} and Figure~\ref{fig_rel5-2}).

\begin{theorem}
\label{rel5}
The following formula holds:
\begin{align}
\label{5th-rel}
\Delta_{K}\cdot \Delta_{H}=\Delta_{K\cap H}\cdot \left(\sum_{i=0}^{k}(-1)^{i}X^{i}\cdot \mathfrak{S}_{k-i}(M_{v}\ |\ v\not\in K\cup H) \right)\in H^{4n}(\mathcal{GQ}_{2n}),
\end{align}
where $k=|K\cap H|-1$. 
\end{theorem}
\begin{proof}
If $K\cap H=\emptyset$, then the statement follows from Relation~1 and $\Delta_{\emptyset}=0$.
So we may assume $K\cap H\not=\emptyset$.

Because $\Delta_{K}, \Delta_{H}\in H^{2n}(\mathcal{GQ}_{2n})$, their multiplication satisfies $\Delta_{K}\cdot \Delta_{H}\in H^{4n}(\mathcal{GQ}_{2n})$.
Moreover, the degree of each term on the right-hand side in \eqref{5th-rel} satisfies that
\begin{align*}
\deg \Delta_{K\cap H}+
\deg X^{i}+\deg \mathfrak{S}_{k-i}(M_{v}\ |\ v\not\in K\cup H) 
=(4n-2k)+2i+2(k-i)=2n.
\end{align*}
 
For every $p\not\in K\cap H$, because $\Delta_{K}\cdot \Delta_{H}(p)=\Delta_{K\cap H}(p)=0$, the relation \eqref{5th-rel} holds. 
 
For $p\in K\cap H$, by the definitions of $\Delta_{K}$'s and $M_{v}$'s, 
it is easy to check that 
\begin{align*}
\Delta_{K}\cdot \Delta_{H}(p)
=\Delta_{K\cap H}(p)\cdot \prod_{v\not\in K\cup H\cup \{\overline{p}\}}M_{v}(p).
\end{align*}
Because $|K\cap H|=k+1$ for $0\le k\le n$, we may put  
\begin{align*}
K\cap H=\{a_{0},a_{1},\ldots, a_{k}\}\subset V_{2n},
\end{align*}
where we assume $p=a_{0}$.
Because $|K|=|H|=n+1=\frac{|V_{2n}|}{2}$, we also have that 
$K^{c}=\{\overline{a}\ |\ a\in K\}$ and $H^{c}=\{\overline{b}\ |\ b\in H\}$.
Therefore, 
$K^{c}\cap H^{c}=\{\overline{x}\ |\ x\in K\cap H\}=\{\overline{a_{0}},\overline{a_{1}},\ldots, \overline{a_{k}}\}$.
This shows that 
$v\not\in K\cup H\cup \{\overline{p}\}$ if and only if $v\in (K\cup H\cup \{\overline{p}\})^{c}=(K^{c}\cap H^{c})\setminus \{\overline{p}\}=\{\overline{a_{1}},\ldots, \overline{a_{k}}\}$.
Therefore, if we put $\mathcal{A}:=\{\overline{a_{1}},\ldots, \overline{a_{k}}\}$, 
\begin{align}
\label{equation-for-1st-step}
\Delta_{K}\cdot \Delta_{H}(p)=\Delta_{K\cap H}(p)\cdot \prod_{v\in \mathcal{A}}M_{v}(p).
\end{align}
On the other hand,
\begin{align}
\label{equation-for-2nd-step}
 \sum_{i=0}^{k}(-1)^{i}X^{i}\cdot \mathfrak{S}_{k-i}(M_{v}\ |\ v\not\in K\cup H)(p)  
=\sum_{i=0}^{k}(-1)^{i}X(p)^{i}\cdot \mathfrak{S}_{k-i}(M_{v}\ |\ v\in \{\overline{p}\}\cup\mathcal{A})(p). 
\end{align}
By the definition of $M_{v}$, 
we have $M_{\overline{p}}(p)=X(p)$.
Therefore, for $0\le i\le k-1$, 
\begin{align*}
& \mathfrak{S}_{k-i}(M_{v}\ |\ v\in \{\overline{p}\}\cup\mathcal{A})(p) \\
&=
\mathfrak{S}_{k-i}(M_{v}\ |\ v\in \mathcal{A})(p)+
X(p)\cdot \mathfrak{S}_{k-i-1}(M_{v}\ |\ v\in \mathcal{A})(p).
\end{align*}
Substituting this into \eqref{equation-for-2nd-step}, we have 
\begin{align*}
& \sum_{i=0}^{k}(-1)^{i}X(p)^{i}\cdot \mathfrak{S}_{k-i}(M_{v}\ |\ v\in \{\overline{p}\}\cup\mathcal{A})(p) \\
=&\sum_{i=0}^{k-1}(-1)^{i}X(p)^{i}\cdot\mathfrak{S}_{k-i}(M_{v}\ |\ v\in \mathcal{A})(p)+
\sum_{i=0}^{k-1}(-1)^{i}X(p)^{i+1}\cdot \mathfrak{S}_{k-i-1}(M_{v}\ |\ v\in \mathcal{A})(p)
+(-1)^{k}X(p)^{k} \\
=&\mathfrak{S}_{k}(M_{v}\ |\ v\in \mathcal{A}\})(p) \\
=&\prod_{v\in \mathcal{A}}M_{v}(p)\quad ({\rm by}\ |\mathcal{A}|=k).
\end{align*}
Combining \eqref{equation-for-1st-step} and \eqref{equation-for-2nd-step}, we obtain \eqref{5th-rel}.
\end{proof}

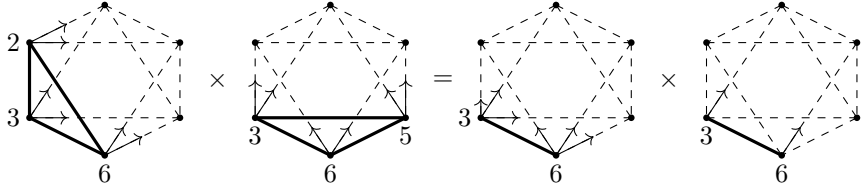
\begin{figure}[H]
\begin{tikzpicture}
\begin{scope}[xscale=0.25, yscale=0.25]
\fill(0,4) circle (5pt);
\fill(-4,2) circle (5pt);
\fill(-4,-2) circle (5pt);
\fill(0,-4) circle (5pt);
\fill(4,-2) circle (5pt);
\fill(4,2) circle (5pt);

\draw[dashed] (0,4)--(-4,2);
\draw[dashed] (0,4)--(-4,-2);
\draw[dashed] (0,4)--(4,-2);
\draw[dashed] (0,4)--(4,2);
\draw[very thick] (-4,2)--(-4,-2);
\draw[very thick] (-4,2)--(0,-4);
\draw[dashed] (-4,2)--(4,2);
\draw[dashed] (-4,-2)--(4,-2);
\draw[very thick] (-4,-2)--(0,-4);
\draw[dashed] (4,2)--(4,-2);
\draw[dashed] (4,2)--(0,-4);
\draw[dashed] (4,-2)--(0,-4);

\draw[->] (-4,2)--(-2,2);
\draw[->] (-4,2)--(-2,3);
\draw[->] (-4,-2)--(-2,-2);
\draw[->] (-4,-2)--(-3,-0.5);
\draw[->] (0,-4)--(1,-2.5);
\draw[->] (0,-4)--(2,-3);

\node[left] at (-4,2) {$2$};
\node[left] at (-4,-2) {$3$};
\node[below] at (0,-4) {$6$};

\node at (6,0) {$\times$};

\fill(12,4) circle (5pt);
\fill(8,2) circle (5pt);
\fill(8,-2) circle (5pt);
\fill(12,-4) circle (5pt);
\fill(16,-2) circle (5pt);
\fill(16,2) circle (5pt);

\draw[dashed] (12,4)--(8,2);
\draw[dashed] (12,4)--(8,-2);
\draw[dashed] (12,4)--(16,-2);
\draw[dashed] (12,4)--(16,2);
\draw[dashed] (8,2)--(8,-2);
\draw[dashed] (8,2)--(12,-4);
\draw[dashed] (8,2)--(16,2);
\draw[very thick] (8,-2)--(16,-2);
\draw[very thick] (8,-2)--(12,-4);
\draw[dashed] (16,2)--(16,-2);
\draw[dashed] (16,2)--(12,-4);
\draw[very thick] (16,-2)--(12,-4);

\draw[->] (8,-2)--(8,0);
\draw[->] (8,-2)--(9,-0.5);
\draw[->] (16,-2)--(15,-0.5);
\draw[->] (16,-2)--(16,0);
\draw[->] (12,-4)--(11,-2.5);
\draw[->] (12,-4)--(13,-2.5);

\node[below] at (8,-2) {$3$};
\node[below] at (12,-4) {$6$};
\node[below] at (16,-2) {$5$};

\node at (18,0) {$=$}; 


\fill(24,4) circle (5pt);
\fill(20,2) circle (5pt);
\fill(20,-2) circle (5pt);
\fill(24,-4) circle (5pt);
\fill(28,-2) circle (5pt);
\fill(28,2) circle (5pt);

\draw[dashed] (24,4)--(20,2);
\draw[dashed] (24,4)--(20,-2);
\draw[dashed] (24,4)--(28,-2);
\draw[dashed] (24,4)--(28,2);
\draw[dashed] (20,2)--(20,-2);
\draw[dashed] (20,2)--(24,-4);
\draw[dashed] (20,2)--(28,2);
\draw[dashed] (20,-2)--(28,-2);
\draw[very thick] (20,-2)--(24,-4);
\draw[dashed] (28,2)--(28,-2);
\draw[dashed] (28,2)--(24,-4);
\draw[dashed] (28,-2)--(24,-4);

\draw[->] (20,-2)--(20,-1);
\draw[->] (20,-2)--(22,-2);
\draw[->] (20,-2)--(21,-0.5);
\draw[->] (24,-4)--(25,-2.5);
\draw[->] (24,-4)--(26,-3);
\draw[->] (24,-4)--(23,-2.5);

\node[left] at (20,-2) {$3$};
\node[below] at (24,-4) {$6$};

\node at (30,0) {$\times$};

\fill(36,4) circle (5pt);
\fill(32,2) circle (5pt);
\fill(32,-2) circle (5pt);
\fill(36,-4) circle (5pt);
\fill(40,-2) circle (5pt);
\fill(40,2) circle (5pt);

\draw[dashed] (36,4)--(32,2);
\draw[dashed] (36,4)--(32,-2);
\draw[dashed] (36,4)--(40,-2);
\draw[dashed] (36,4)--(40,2);
\draw[dashed] (32,2)--(32,-2);
\draw[dashed] (32,2)--(36,-4);
\draw[dashed] (32,2)--(40,2);
\draw[dashed] (32,-2)--(40,-2);
\draw[very thick] (32,-2)--(36,-4);
\draw[dashed] (40,2)--(40,-2);
\draw[dashed] (40,2)--(36,-4);
\draw[dashed] (40,-2)--(36,-4);

\draw[->] (32,-2)--(33,-0.5);
\draw[->] (36,-4)--(37,-2.5);

\node[below] at (32,-2) {$3$};
\node[below] at (36,-4) {$6$};

\end{scope}
\end{tikzpicture}
\caption{
This represents the following relation (also see Figure~\ref{fig_rel5-2}):
\begin{align*}
\Delta_{\{2,3,6\}}\cdot \Delta_{\{3,5,6\}} 
=\Delta_{\{3,6\}}\cdot (\mathfrak{S}_{1} (M_{1}, M_{4})-X) 
=\Delta_{\{3,6\}}\cdot (M_{1}+ M_{4}-X), 
\end{align*}
because $K\cap H=\{3,6\}$ and $K\cup H=\{2,3,5,6\}\subset I$ (so $V_{4}\setminus (K\cup H)=\{1,4\}$).}
\label{fig_rel5}
\end{figure}

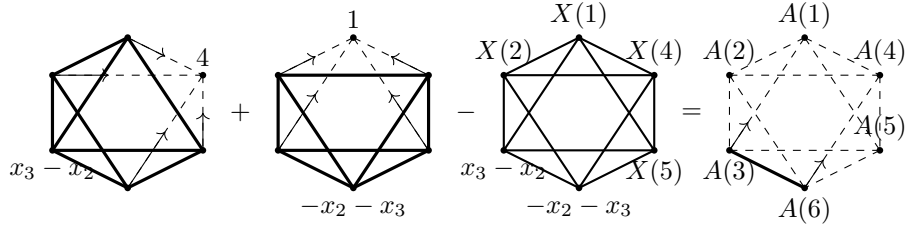
\begin{figure}[H]
\begin{tikzpicture}
\begin{scope}[xscale=0.25, yscale=0.25]
\fill(0,4) circle (5pt);
\fill(-4,2) circle (5pt);
\fill(-4,-2) circle (5pt);
\fill(0,-4) circle (5pt);
\fill(4,-2) circle (5pt);
\fill(4,2) circle (5pt);

\draw[very thick] (0,4)--(-4,2);
\draw[very thick] (0,4)--(-4,-2);
\draw[very thick] (0,4)--(4,-2);
\draw[dashed] (0,4)--(4,2);
\draw[very thick] (-4,2)--(-4,-2);
\draw[very thick] (-4,2)--(0,-4);
\draw[dashed] (-4,2)--(4,2);
\draw[very thick] (-4,-2)--(4,-2);
\draw[very thick] (-4,-2)--(0,-4);
\draw[dashed] (4,2)--(4,-2);
\draw[dashed] (4,2)--(0,-4);
\draw[very thick] (4,-2)--(0,-4);

\draw[->] (-4,2)--(-1,2);
\draw[->] (0,4)--(2,3);
\draw[->] (4,-2)--(4,0);
\draw[->] (0,-4)--(2,-1);

\node at (6,0) {$+$};

\fill(12,4) circle (5pt);
\fill(8,2) circle (5pt);
\fill(8,-2) circle (5pt);
\fill(12,-4) circle (5pt);
\fill(16,-2) circle (5pt);
\fill(16,2) circle (5pt);

\draw[dashed] (12,4)--(8,2);
\draw[dashed] (12,4)--(8,-2);
\draw[dashed] (12,4)--(16,-2);
\draw[dashed] (12,4)--(16,2);
\draw[very thick] (8,2)--(8,-2);
\draw[very thick] (8,2)--(12,-4);
\draw[very thick] (8,2)--(16,2);
\draw[very thick] (8,-2)--(16,-2);
\draw[very thick] (8,-2)--(12,-4);
\draw[very thick] (16,2)--(16,-2);
\draw[very thick] (16,2)--(12,-4);
\draw[very thick] (16,-2)--(12,-4);

\draw[->] (8,2)--(10,3);
\draw[->] (8,-2)--(10,1);
\draw[->] (16,-2)--(14,1);
\draw[->] (16,2)--(14,3);

\node[above] at (12,4) {$1$};
\node[above] at (4,2) {$4$};
\node[below] at (12,-4) {$-x_{2}-x_{3}$};
\node[below] at (-4,-2) {$x_{3}-x_{2}$};

\node at (18,0) {$-$};

\fill(24,4) circle (5pt);
\node[above] at (24,4) {$X(1)$};
\fill(20,2) circle (5pt);
\node[above] at (20,2) {$X(2)$};
\fill(20,-2) circle (5pt);
\node[below] at (20,-2) {$x_{3}-x_{2}$};
\fill(24,-4) circle (5pt);
\node[below] at (24,-4) {$-x_{2}-x_{3}$};
\fill(28,-2) circle (5pt);
\node[below] at (28,-2) {$X(5)$};
\fill(28,2) circle (5pt);
\node[above] at (28,2) {$X(4)$};

\draw[thick] (24,4)--(20,2);
\draw[thick] (24,4)--(20,-2);
\draw[thick] (24,4)--(28,-2);
\draw[thick] (24,4)--(28,2);
\draw[thick] (20,2)--(20,-2);
\draw[thick] (20,2)--(24,-4);
\draw[thick] (20,2)--(28,2);
\draw[thick] (20,-2)--(28,-2);
\draw[thick] (20,-2)--(24,-4);
\draw[thick] (28,2)--(28,-2);
\draw[thick] (28,2)--(24,-4);
\draw[thick] (28,-2)--(24,-4);

\node at (30,0) {$=$};


\fill(36,4) circle (5pt);
\fill(32,2) circle (5pt);
\fill(32,-2) circle (5pt);
\fill(36,-4) circle (5pt);
\fill(40,-2) circle (5pt);
\fill(40,2) circle (5pt);

\draw[dashed] (36,4)--(32,2);
\draw[dashed] (36,4)--(32,-2);
\draw[dashed] (36,4)--(40,-2);
\draw[dashed] (36,4)--(40,2);
\draw[dashed] (32,2)--(32,-2);
\draw[dashed] (32,2)--(36,-4);
\draw[dashed] (32,2)--(40,2);
\draw[dashed] (32,-2)--(40,-2);
\draw[very thick] (32,-2)--(36,-4);
\draw[dashed] (40,2)--(40,-2);
\draw[dashed] (40,2)--(36,-4);
\draw[dashed] (40,-2)--(36,-4);

\draw[->] (32,-2)--(33,-0.5);
\draw[->] (36,-4)--(37,-2.5);

\node[below] at (32,-2) {$A(3)$};
\node[below] at (36,-4) {$A(6)$};
\node[above] at (36, 4) {$A(1)$};
\node[above] at (32, 2) {$A(2)$};
\node[above] at (40, 2) {$A(4)$};
\node[above] at (40, -2) {$A(5)$};

\end{scope}
\end{tikzpicture}
\caption{This figure represents the term $A=M_{1}+M_{4}-X$ in Figure~\ref{fig_rel5}.
Note that $A(3)=A(6)=-x_{2}$ by Figure~\ref{fig_0-cochain_presentation}.
Moreover, 
$A(1), A(2), A(4), A(5)$ might not be $0\in H^{2}(BT^{3})$; however, $\Delta_{\{3,6\}}(1)=\Delta_{\{3,6\}}(2)=\Delta_{\{3,6\}}(4)=\Delta_{\{3,6\}}(5)=0$.}
\label{fig_rel5-2}
\end{figure}


\section{Comparison of two ordinary cohomology rings $H^{*}(Q_{4n})$ and $H^{*}(Q_{4n+2})$}
\label{sect:7}

Since $H^{odd}(Q_{2n})=0$ by \cite{La74}, $Q_{2n}$ is the equivariantly formal GKM manifold (see \cite{GKM}).
Therefore, its ordinary cohomology also can be computed by the quotient of $H_{T}^{*}(Q_{2n})$ by $H^{>0}(BT^{n+1})$.
Thus, by using Theorem~\ref{main-theorem2} and Proposition~\ref{polynomial generators}, we also have the ordinary cohomology of $Q_{2n}$ by the different way of \cite{La74}.
\begin{corollary}
\label{main-corollary}
The ordinary cohomology $H^{*}(Q_{2n})$ is isomorphic to $\mathbb{Z}[\mathcal{GQ}_{2n}]/\mathcal{J}$, where $\mathcal{J}$ is generated by
\begin{align*}
M_{i+1}-M_{1}
\end{align*}
for $i=1,\ldots, n+1$.
\end{corollary}

Recall that the cohomology ring formula of $Q_{2n}$ depends on $n$ is even or odd, i.e., by \cite{La74}, 
\begin{align*}
& H^{*}(Q_{4n})\simeq \mathbb{Z}[c,x]/\langle c^{2n+1}-2cx, x^{2}- c^{2n}x \rangle; \\
& H^{*}(Q_{4n+2})\simeq \mathbb{Z}[c,x]/\langle c^{2n+2}-2cx, x^{2} \rangle.
\end{align*}
In this final section, 
we give the combinatorial reason why this difference occurs by using 
Corollary~\ref{main-corollary}.
To do that, 
the following lemma is essential.

\begin{lemma}
\label{final-lemma}
If $K\subset V_{2n}$ is the following subset with the property $(\ast)$:
\begin{align*}
K=\{i_{1},\ldots, i_{n+1}\}.
\end{align*}
Then, there is
the following formula in $\mathbb{Z}[\mathcal{GQ}_{2n}]/\mathcal{J}$:
\begin{align*}
\Delta_{K}=\Delta_{\{i_{1},\ldots,i_{n-1},\overline{i_{n}},\overline{i_{n+1}}\}}.
\end{align*}
\end{lemma} 
\begin{proof}
By definition of $\mathcal{J}$, 
in $\mathbb{Z}[\mathcal{GQ}_{2n}]/\mathcal{J}$,
we have 
\begin{align*}
M_{1}=M_{2}=\cdots =M_{n+1}=M_{n+2}.
\end{align*}
By Relation~2, $M_{n+1}+M_{n+2}=M_{i}+M_{\overline{i}}$ for all $i=1,\ldots, n$.
Therefore, we also have that 
\begin{align*}
M_{1}=M_{n+3}=\cdots =M_{2n+2}.
\end{align*}
Consequently, we have $M_{i}=M_{j}$ for all $i,j\in [2n+2]$.
Because $K$ satisfies the property $(\ast)$, 
 for every $a\in K$, the subset $I:=K\setminus\{a\}\subset V_{2n}$ satisfies that $|I|=n$; moreover, there are unique pair $\{a,\overline{a}\}\subset V_{2n}\setminus I$.
Therefore, we may apply Relation~3 for $K\setminus\{a\}$.
Together with $M_{1}=M_{i}$ for all $i\in V_{2n}$ as above, we have 
\begin{align}
\label{final-relation}
M_{1}^{n}=\Delta_{K^{c}}+\Delta_{((K\setminus\{a\})\cup \{\overline{a}\})^{c}}&=\Delta_{K^{c}}+\Delta_{(K^{c}\setminus\{\overline{a}\})\cup \{a\}}.
\end{align}
Note that this equation \eqref{final-relation} holds for all $K\subset V_{2n}$ which satisfies the assumption of this lemma.
Therefore, we have 
\begin{align*}
M_{1}^{n}&=\Delta_{\{i_{1},\ldots, i_{n+1}\}}+\Delta_{\{i_{1},\ldots, i_{n}, \overline{i_{n+1}}\}}=\Delta_{\{i_{1},\ldots, i_{n}, \overline{i_{n+1}}\}}+\Delta_{\{i_{1},\ldots,i_{n-1}, \overline{i_{n}},\overline{i_{n+1}}\}}.
\end{align*}
Thus, 
$\Delta_{\{i_{1},\ldots, i_{n+1}\}}=\Delta_{\{i_{1},\ldots,i_{n-1}, \overline{i_{n}},\overline{i_{n+1}}\}}$.
\end{proof}
Consequently, we have the following corollary.
\begin{corollary}
\label{final-cor}
For $K\subset V_{2m}$ which satisfies the assumption of Lemma~\ref{final-lemma}, 
 there are the following relations:
\begin{itemize}
\item $\Delta_{K}\cdot (M_{i}^{2n}-\Delta_{K})=0$ in $\mathbb{Z}[\mathcal{GQ}_{4n}]/\mathcal{J}$ if $m=2n$;
\item $\Delta_{K}^{2}=0$ in $\mathbb{Z}[\mathcal{GQ}_{4n+2}]/\mathcal{J}$ if $m=2n+1$.
\end{itemize}
\end{corollary}
\begin{proof}
Suppose $m=2n+1$, i.e., $m\equiv 1 \mod 2$. 
By iterating to use Lemma~\ref{final-lemma}, we have $\Delta_{K^{c}}=\Delta_{K}$.
Therefore, by $K\cap K^{c}=\emptyset$ and   
Relation~1, we have the 2nd relation in the statement. 

Suppose $m=2n$, i.e., $m\equiv 0 \mod 2$. 
In this case, $K=\{i_{1},\ldots, i_{2n+1}\}$.
By iterating to use Lemma~\ref{final-lemma}, we obtain  
\begin{align*}
\Delta_{K}=\Delta_{\{i_{1},\overline{i_{2}}, \ldots, \overline{i_{2n+1}}\}}.
\end{align*}
By applying \eqref{final-relation} to $i_{1}\in \{\overline{i_{1}},i_{2}\ldots, i_{2n+1}\}^{c}=\{i_{1},\overline{i_{2}}\ldots, \overline{i_{2n+1}}\}$, it follows from this relation that 
\begin{align*}
M_{1}^{2n}&=
\Delta_{\{i_{1},\overline{i_{2}}, \ldots, \overline{i_{2n+1}}\}}+\Delta_{\{\overline{i_{1}},\overline{i_{2}}, \ldots, \overline{i_{2n+1}}\}}=\Delta_{K}+\Delta_{K^{c}}.
\end{align*}
Hence, we have $\Delta_{K^{c}}=M_{1}^{2n}-\Delta_{K}$. 
Therefore, by $K\cap K^{c}=\emptyset$ and Relation~1, we have the 1st relation in the statement. 
\end{proof}

Figure~\ref{final-figure1} and Figure~\ref{final-figure2} show the difference of the ordinary cohomology for $n=2$ and $n=3$.
For example, in Figure~\ref{final-figure1}, $K=\{3,5,6\}$ (see Figure~\ref{2-examples}).
In this case, by applying Lemma~\ref{final-lemma}, $\Delta_{K}=\Delta_{H}$ for $H=\{1,2,3\}=\{2,4,6\}=\{1,4,5\}$.
This shows that $K\cap H\not=\emptyset$; therefore, $\Delta_{K}^{2}\not=0$.
However, in Figure~\ref{final-figure2}, we can take such $H$ as $K^{c}$; this gives $\Delta_{K}^{2}=0$

\begin{figure}[H]
\begin{tabular}{cc}
\begin{minipage}[t]{0.5\linewidth}
\begin{tikzpicture}
\begin{scope}[xscale=0.25, yscale=0.25]
\fill(0,4) circle (5pt);
\node[above] at (0,4) {$\Delta_{K}$}; 
\fill(-4,2) circle (5pt);
\fill(-4,-2) circle (5pt);
\fill(0,-4) circle (5pt);
\fill(4,-2) circle (5pt);
\fill(4,2) circle (5pt);

\draw[dashed] (0,4)--(-4,2);
\draw[dashed] (0,4)--(-4,-2);
\draw[dashed] (0,4)--(4,-2);
\draw[dashed] (0,4)--(4,2);
\draw[dashed] (-4,2)--(-4,-2);
\draw[dashed] (-4,2)--(0,-4);
\draw[dashed] (-4,2)--(4,2);
\draw[very thick] (-4,-2)--(4,-2);
\draw[very thick] (-4,-2)--(0,-4);
\draw[dashed] (4,2)--(4,-2);
\draw[dashed] (4,2)--(0,-4);
\draw[very thick] (4,-2)--(0,-4);

\node at (6,0) {$\times$};

\fill(12,4) circle (5pt);
\node[above] at (12,4) {$\Delta_{K^{c}}=M_{1}^{2}-\Delta_{K}$};
\fill(8,2) circle (5pt);
\fill(8,-2) circle (5pt);
\fill(12,-4) circle (5pt);
\fill(16,-2) circle (5pt);
\fill(16,2) circle (5pt);

\draw[very thick] (12,4)--(8,2);
\draw[dashed] (12,4)--(8,-2);
\draw[dashed] (12,4)--(16,-2);
\draw[very thick] (12,4)--(16,2);
\draw[dashed] (8,2)--(8,-2);
\draw[dashed] (8,2)--(12,-4);
\draw[very thick] (8,2)--(16,2);
\draw[dashed] (8,-2)--(16,-2);
\draw[dashed] (8,-2)--(12,-4);
\draw[dashed] (16,2)--(16,-2);
\draw[dashed] (16,2)--(12,-4);
\draw[dashed] (16,-2)--(12,-4);

\node at (18,0) {$=0$};
\end{scope}
\end{tikzpicture}
\caption{$\Delta_{K}(M_{1}^{2}-\Delta_{K})=0$ for $n=2$.}
\label{final-figure1}
\end{minipage} &
\begin{minipage}[t]{0.5\linewidth}
\begin{tikzpicture}
\begin{scope}[xscale=0.25, yscale=0.25]
\coordinate (1) at (-1.5,4);
\coordinate (2) at (-4,1.5);
\coordinate (3) at (-4,-1.5);
\coordinate (4) at (-1.5,-4);
\coordinate (5) at (1.5,4);
\coordinate (6) at (4,1.5);
\coordinate (7) at (4,-1.5);
\coordinate (8) at (1.5,-4);

\node[above] at (0,4) {$\Delta_{K}$}; 

\fill(1) circle (5pt);
\fill(2) circle (5pt);
\fill(3) circle (5pt);
\fill(4) circle (5pt);
\fill(5) circle (5pt);
\fill(6) circle (5pt);
\fill(7) circle (5pt);
\fill(8) circle (5pt);

\draw[dashed] (1)--(2);
\draw[dashed] (1)--(3);
\draw[dashed] (1)--(4);
\draw[dashed] (1)--(5);
\draw[dashed] (1)--(6);
\draw[dashed] (1)--(7);
\draw[dashed] (2)--(3);
\draw[dashed] (2)--(4);
\draw[dashed] (2)--(5);
\draw[dashed] (2)--(6);
\draw[dashed] (2)--(8);
\draw[very thick] (3)--(4);
\draw[dashed] (3)--(5);
\draw[very thick] (3)--(7);
\draw[very thick] (3)--(8);
\draw[dashed] (4)--(6);
\draw[very thick] (4)--(7);
\draw[very thick] (4)--(8);
\draw[dashed] (5)--(6);
\draw[dashed] (5)--(7);
\draw[dashed] (5)--(8);
\draw[dashed] (6)--(7);
\draw[dashed] (6)--(8);
\draw[very thick] (7)--(8);

\node at (6,0) {$\times$};

\coordinate (a1) at (10.5,4);
\coordinate (a2) at (8,1.5);
\coordinate (a3) at (8,-1.5);
\coordinate (a4) at (10.5,-4);
\coordinate (a5) at (13.5,4);
\coordinate (a6) at (16,1.5);
\coordinate (a7) at (16,-1.5);
\coordinate (a8) at (13.5,-4);

\node[above] at (12,4) {$\Delta_{K^{c}}=\Delta_{K}$};

\fill(a1) circle (5pt);
\fill(a2) circle (5pt);
\fill(a3) circle (5pt);
\fill(a4) circle (5pt);
\fill(a5) circle (5pt);
\fill(a6) circle (5pt);
\fill(a7) circle (5pt);
\fill(a8) circle (5pt);

\draw[very thick] (a1)--(a2);
\draw[dashed] (a1)--(a3);
\draw[dashed] (a1)--(a4);
\draw[very thick] (a1)--(a5);
\draw[very thick] (a1)--(a6);
\draw[dashed] (a1)--(a7);
\draw[dashed] (a2)--(a3);
\draw[dashed] (a2)--(a4);
\draw[very thick] (a2)--(a5);
\draw[very thick] (a2)--(a6);
\draw[dashed] (a2)--(a8);
\draw[dashed] (a3)--(a4);
\draw[dashed] (a3)--(a5);
\draw[dashed] (a3)--(a7);
\draw[dashed] (a3)--(a8);
\draw[dashed] (a4)--(a6);
\draw[dashed] (a4)--(a7);
\draw[dashed] (a4)--(a8);
\draw[very thick] (a5)--(a6);
\draw[dashed] (a5)--(a7);
\draw[dashed] (a5)--(a8);
\draw[dashed] (a6)--(a7);
\draw[dashed] (a6)--(a8);
\draw[dashed] (a7)--(a8);

\node at (18,0) {$=0$};
\end{scope}
\end{tikzpicture}
\caption{$\Delta_{K}^{2}=0$ for $n=3$.}
\label{final-figure2}
\end{minipage}
\end{tabular}
\end{figure}

\section*{Acknowledgment}
The author was supported by JSPS KAKENHI Grant Number 21K03262.
He also would like to thank Professor Fuichi Uchida (1938--2021) who 
was the supervisor in his master course. 
He learned a lot from Professor Uchida including complex quadrics. 
He also would like to thank the anonymous referees for their careful reading of the previous manuscript and their comments and suggestions.



\begin{thebibliography}{99}
\bibitem{DKS}
A.~Darby, S.~Kuroki and J.~Song, 
\emph{Equivariant cohomology of torus orbifolds}, 
Canadian J.~of Math., \textbf{74}, Issue 2 (2022), 299--328.

\bibitem{EH13}
D.~Eisenbud and J.~Harri,
3264 and All That: Intersection Theory in Algebraic Geometry, 
Cambridge University Press, Cambridge, 2013.

\bibitem{EKM08}
R.~Elman, N.~Karpenko and A.~Merkurjev,
The algebraic and geometric theory of quadratic forms. 
Amer.~Math.~Soc.~Colloquium Publ., \textbf{56}. 
American Mathematical Society, Providence, RI, 2008.

\bibitem{FP07}
M.~Franz and V.~Puppe, 
\emph{Exact cohomology sequences with integral coefficients for torus actions}, Transform.~Groups \textbf{12} (2007), no. 1, 65--76.

\bibitem{FY}
M.~Franz and H.~Yamanaka, 
\emph{Graph equivariant cohomological rigidity for GKM graphs},
Proc. Japan Acad. Ser. A Math. Sci. 95 (2019), 107--110.

\bibitem{FIM}
Y.~Fukukawa, H.~Ishida and M.~Masuda, 
\emph{The cohomology ring of the GKM graph of a flag manifold of classical
type}, Kyoto J.~Math.~\textbf{54} (2014), No. 3, 653--677.

\bibitem{GHZ}
V.~Guillemin, T.~Holm and C.~Zara, 
\emph{A GKM description of the equivariant cohomology ring of a homogeneous
space}, J. Algebraic Combin.~\textbf{23} (2006) no. 1, 21--41.

\bibitem{GKM}
M. Goresky, R. Kottwitz and R. MacPherson, 
\emph{Equivariant cohomology, Koszul duality, and the localization theorem},
Invent.~Math.~\textbf{131} (1998), 25--83.

\bibitem{GKZ}
O.~Goertsches, P.~Konstantis and L.~Zoller, 
\emph{Realization of GKM fibrations and new examples of Hamiltonian non-K{\rm$\ddot{a}$}hler actions}.
arXiv:2003.11298.

\bibitem{GSZ}
V.~Guillemin, S.~Sabatini and C.~Zara,
\emph{Cohomology of GKM fiber bundles},
J.~of Alg.~Comb.~vol 35, 19--59 (2012).

\bibitem{GZ01}
V.~Guillemin and C.~Zara, 
\emph{One-skeleta, Betti numbers, and equivariant cohomology}, 
Duke Math.~J.~\textbf{107} 2 (2001), 283--349.

\bibitem{Jo}
M.~Jovanovi\'c,
\emph{On integral cohomology algebra of some oriented Grassmann manifolds},
arXiv:2212.08557.

\bibitem{Ku}
S.~Kuroki,
\emph{Classification of compact transformation groups on complex quadrics with codimension one orbits}, 
Osaka J.~Math., \textbf{46} 1, (2009), 21--85.

\bibitem{Ku09}
S.~Kuroki, 
\emph{Introduction to GKM theory}, 
Trends in Mathematics - New Series \textbf{11} 2 (2009), 113--129.

\bibitem{Ku16}
S.~Kuroki, 
\emph{An Orlik-Raymond type classification of simply connected $6$-dimensional torus manifolds with vanishing odd degree cohomology}, 
Pacific J.~of Math.~\textbf{280} (2016), 89--114.

\bibitem{Ku19}
S.~Kuroki,
\emph{Upper bounds for the dimension of tori acting on GKM manifolds}, 
J.~of the Math.~Soc.~Japan, \textbf{71} (2019), 483--513.

\bibitem{KS}
S.~Kuroki and G.~Solomadin, 
\emph{Borel-Hirzebruch type formula for graph equivariant cohomology of projective bundle over GKM-graph}, arXiv:2207.11380.

\bibitem{KU}
S.~Kuroki and V.~Uma, 
\emph{GKM graph locally modeled by $T^{n}\times S^{1}$-action on $T^{*}\mathbb{C}^{n}$ and its graph equivariant cohomology}, arXiv:2106.11598, to appear in The Fields Institute Communication Volume: Toric Topology and Polyhedral Products.

\bibitem{La72}
H.~Lai,
\emph{Characteristic classes of real manifolds immersed in complex manifolds},
Trans.~of Amer.~Math.~Soc.~\textbf{172}, (1972), 1--33.

\bibitem{La74}
H.~Lai,
\emph{On the topology of the even-dimensional complex quadrics},
Proc.~Amer.~Math.~Soc.~\textbf{46}, No.~3, (1974), 419--425.

\bibitem{Ma99}
M.~Masuda, 
\emph{Unitary toric manifolds, multi-fans and equivariant index}, 
Tohoku Math.~J.~\textbf{51} (1999), 237--265

\bibitem{MMP}
H.~Maeda, M.~Masuda and T.~Panov, 
\emph{Torus graphs and simplicial posets},
Adv.\ Math.\ \textbf{212} (2007), 458--483.

\bibitem{MP}
M.~Masuda and T.~Panov, 
\emph{On the cohomology of torus manifolds}, 
Osaka J.~Math., \textbf{43}, no. 3 (2006), 711--746.

\bibitem{Se06}
J.~Seade, 
On the Topology of Isolated Singularities in Analytic Spaces, 
Progress in Mathematics, 241. Birkh${\rm \ddot{a}}$user Verlag, Basel, 2006.

\bibitem{Uc77}
F.~Uchida,
\emph{Classification of compact transformation groups on cohomology complex projective spaces with codimension one orbits},
Japan J.~Math \textbf{3}, (1977), 141--189.

\end{thebibliography}
\end{document}